\documentclass{amsart}
\usepackage[latin1]{inputenc}
\usepackage[all]{xy}
\usepackage{amssymb}
\usepackage{amsmath}
\usepackage{graphicx}
\usepackage{enumitem}
\usepackage{epsfig}
\newtheorem{rem}{Remark}[section]
\newtheorem{proposition}[rem]{Proposition}
\newtheorem{prop}[rem]{Proposition}
\newtheorem{lemma}[rem]{Lemma}
\newtheorem{corollary}[rem]{Corollary}
\newtheorem{defi}[rem]{Definition}
\newtheorem{theorem}[rem]{Theorem}

\newcommand{\ra}{\rightarrow}

\newcommand{\PP}{ \mathbb{P}}
\newcommand{\C }{ \mathbb{C}}

\newcommand{\Z}{\mathbb{Z}}
\newcommand{\N}{\mathbb{N}}
\newcommand{\Q}{\mathbb{Q}}

\newcommand{\of}{\mathcal{O}}
\DeclareMathOperator{\NS}{\rm NS}
\DeclareMathOperator{\Def}{\rm Def}
\DeclareMathOperator{\ch}{\rm ch}
\DeclareMathOperator{\td}{\rm Td}

\DeclareMathOperator{\Sing}{\rm Sing}
\DeclareMathOperator{\rk}{\rm rank}
\DeclareMathOperator{\id}{\rm id}
\DeclareMathOperator{\Aut}{\rm Aut}
\DeclareMathOperator{\Fix}{\rm Fix}
\DeclareMathOperator{\diag}{\rm diag}
\DeclareMathOperator{\Sym}{\rm Sym}
\DeclareMathOperator{\Pic}{\rm Pic}
\newcommand{\hsk}{K3}
\newcommand{\hskn}{K3^{\left[n\right]}}
\newcommand{\hk}{hyperk\"ahler }
\title{Calabi--Yau quotients of hyperk\"ahler four-folds}
\author{Chiara Camere}
\thanks{the first named author was partially supported by Vigoni Project 2012-2013 and PRIN 2010-2011: ``Geometria delle Variet\`a Algebriche''} 
\author{Alice Garbagnati}
\thanks{the second named author was partially supported by PRIN 2010-2011: ``Geometria delle Variet\`a Algebriche'' and FIRB 2012 ``Moduli spaces and their applications''}
\author{Giovanni Mongardi}
\thanks{the third named author was supported by FIRB 2012 ``Moduli spaces and their applications''}
\address{Chiara Camere, Universit\`a degli Studi di Milano,
Dipartimento di Matematica,
Via Cesare Saldini 50,
20133 Milano, Italy} 
\email{chiara.camere@unimi.it}
\urladdr{http://www.mat.unimi.it/users/camere/en/index.html}

\address{Alice Garbagnati, Universit\`a degli Studi di Milano, Dipartimento di Matematica, via Cesare Saldini 50 20133 Milano, Italy }
\email{alice.garbagnati@unimi.it}
\urladdr{https://sites.google.com/site/alicegarbagnati/}

\address{Giovanni Mongardi, Universit\`a degli Studi di Milano, Dipartimento di Matematica, via Cesare Saldini 50 20133 Milano, Italy }
\email{giovanni.mongardi@unimi.it}


\begin{document}

\begin{abstract}
The aim of this paper is to construct Calabi--Yau 4-folds as crepant resolutions of the quotients of a hyperk\"ahler 4-fold $X$ by a non symplectic involution $\alpha$. We first compute the Hodge numbers of a Calabi--Yau constructed in this way in a general setting and then we apply the results to several specific examples of  non symplectic involutions, producing Calabi--Yau 4-folds with different Hodge diamonds. Then we restrict ourselves to the case where $X$ is the Hilbert scheme of two points on a K3 surface $S$ and the involution $\alpha$ is induced by a non symplectic involution on the K3 surface. In this case we compare the Calabi--Yau 4-fold $Y_S$, which is the crepant resolution of $X/\alpha$, with the Calabi--Yau 4-fold $Z_S$, constructed from $S$  through the Borcea--Voisin construction. We give several explicit geometrical examples of both these Calabi--Yau 4-folds describing maps related to interesting linear systems as well as a rational $2:1$ map from $Z_S$ to $Y_S$.
\end{abstract}
\maketitle
\section{Introduction}\label{intro}
By the famous decomposition theorem of Beauville \cite{beau83} and Bogomolov \cite{bogo}, the Ricci flat varieties decompose, after an \'etale cover, into the product of three fundamental building blocks: Abelian varieties, \hk manifolds and Calabi--Yau manifolds. The aim of this paper is to construct a relation between two of these blocks in dimension 4, indeed our starting point is the observation that the presence of a non symplectic involution $\alpha$ on a \hk 4-fold $X$ allows one to construct 
a Calabi--Yau 4-fold as crepant resolution $\widetilde{X/\alpha}$ of the quotient $X/\alpha$.

We observe that several quotients of \hk varieties have been deeply investigated both in the case of symplectic and non symplectic actions and in particular in low dimension. In dimension 2, Calabi--Yau varieties and \hk varieties collapse to the same class of surfaces, the K3 surfaces. In the case of automorphisms acting on K3 surfaces, it is well-known that the quotient by a symplectic automorphism gives, after a minimal resolution, again a K3 surface. This is known not to be the case in general for higher dimensional hyperk\"ahler manifolds; indeed, given a \hk variety with a symplectic automorphism $\alpha$, there is in general no resolution of $X/\sigma$ where the natural quotient symplectic form on the smooth locus is preserved, and partial resolutions give singular irreducible holomorphic symplectic manifolds (see for example \cite{Fujiki}). Up to now, the only known case where a symplectic resolution of the quotient exists is described in \cite{kawat}. So generically, the quotient of a \hk variety by a finite 
automorphism does not produce a \hk variety, but, as we noticed above, it can produce, under some conditions, a Calabi--Yau variety and this is the main topic of this paper.

In Section \ref{sec: general-teo} we consider the quotients of a \hk variety $X$, of dimension $2n$, by a finite automorphism $\alpha$, of prime order $p$. We ask when it is possible to obtain a quotient $X/\alpha$ which has trivial canonical bundle and when it is possible to construct a crepant resolution of $X/\alpha$ by blowing up its singular locus. The main result of this section is Theorem \ref{prop: main-thm higher dim}, where we state that the good candidates are \hk varieties of dimension $2p$ which admit a non symplectic automorphism of order $p$. We also observe that there is a condition on the dimension of the components of the fixed locus, which has to be $p$. This condition automatically excludes the natural non symplectic automorphisms of order $p$ on $S^{[p]}$ unless $p=2$. This is one of the motivations for our attention to the case $p=2$. So we restrict our attention to \hk variety of dimension 4 admitting non symplectic involutions. The study of non symplectic involutions on \hk variety is 
the topic of several papers: in \cite{BeauvilleInvolutions} a topological classification is given, in \cite{bcs14} and \cite{mtwkum} a lattice theoretical classification of automorphisms on two different type of \hk 4-folds is presented with many explicit examples. Other explicit examples are given in \cite{ogr06} and \cite{Ferretti}, \cite{mw}, \cite{ow}, \cite{ikkr}.

All these works provide a large set of explicit examples of non symplectic involutions $\alpha$ defined on \hk 4-folds $X$, and thus one is able to effectively construct Calabi--Yau 4-folds as quotients.

In Section \ref{sec: quotients of hyperkahler} we consider a hyperk\"ahler 4-fold $X$ with a non symplectic involution $\alpha$, and we observe that a crepant resolution of $X/\alpha$ is simply given by blowing up the singular locus. If one knows the action of the non symplectic involution $\alpha$ both on the cohomology of $X$ and on $X$ (more precisely if one knows the topology of the fixed locus of $\alpha$ on $X$), one is able to compute the Hodge numbers of the Calabi--Yau $\widetilde{X/\alpha}$. This is done in the general context in Theorem \ref{prop: hodge diamonds in general} and in some specific examples in Section \ref{subsec: generalized Kummer}, Proposition \ref{prop: Hodge diamond Y for K3 type, also non natural} and Section \ref{subsec Beauville non natural}.

Then in Section \ref{sec: natural symplectic involution} we restrict ourselves to some of the best known \hk 4-folds: denoted by $S$ a K3 surface, the Hilbert scheme of two points of $S$, $S^{[2]}$, is a \hk 4-fold. The 4-fold $S^{[2]}$ is the blow up of $\left(S\times S\right)/\sigma$ in its singular locus, where $\sigma\in\Aut(S\times S)$ is the automorphism switching the two copies of $S$. If $S$ admits a non symplectic involution $\iota_S$, then the involution $\iota_S\times \iota_S\in\Aut(S\times S)$ induces a non symplectic involution on $S^{[2]}$, denoted by $\iota_S^{[2]}$ and is called the natural non symplectic involution of $(S,\iota_S)$. By the construction described above this allows one to produce Calabi--Yau 4-folds, denoted by $Y_S$, as crepant resolutions of $S^{[2]}/\iota_S^{[2]}$ and to compute 
their Hodge numbers which depend only on the topological properties of the fixed locus of $\iota_S$ on $S$, as shown in Theorem \ref{prop: HOdge diamond natural}. Since it is quite a natural question, we have to remark here that neither the mirror symmetry at the level of the K3 surface $S$ nor the lattice theoretic mirror symmetry at the level of the \hk 4-fold $S^{[2]}$ produces a Calabi--Yau 4-fold which is mirror symmetric to $Y_S$ (see Section \ref{subsec: remarks on mirror symmetry}).

Essentially, we produce a Calabi--Yau 4-fold, $Y_S$, by the data $(S,\iota_S)$, where $S$ is a K3 surface and $\iota_S$ is a non symplectic involution acting on it.
On the other hand, there is a very well known and natural  way to produce a (different) Calabi--Yau 4-fold starting from these data, which is the Borcea--Voisin construction, cf.\ \cite{Cynk-Hulek} and \cite{Dillies}. This construction is recalled in Section \ref{subsec: ZS} and reduces in our case to the blow up of $(S\times S)/(\iota_S\times \iota_S)$ in its singular locus, producing another smooth Calabi--Yau 4-fold, denoted in the sequel by $Z_S$. By our construction, one immediately finds that there is a $2:1$ rational map  $Z_S\dashrightarrow Y_S$, indeed $Y_S$ is birational to $\left(S\times S\right)/\langle \sigma, \iota_S\times \iota_S\rangle$ and $Z_S$ is birational to $\left(S\times S\right)/\langle  \iota_S\times \iota_S\rangle$, so the covering involution of the $2:1$ map $Z_S\dashrightarrow Y_S$ is induced on $Z_S$ by the action of $\sigma$ on $S\times S$, as shown in Section \ref{subsec: quotient of ZS}. So we prove that $Y_S$ is a Calabi--Yau 4-fold which is 2:1 covered both by a \hk 4-fold 
and by a Calabi--Yau 4-fold and in fact it has a bidouble cover which is $S\times S$. Since the construction of $Y_S$ is quite explicit, we are also able to describe a $\Q$-basis of its Picard lattice and to identify two 2-divisible divisors: one which is associated to the double cover $\widetilde{S^{[2]}}\ra Y_S$ where $\widetilde{S^{[2]}}$ is the blow up of $S^{[2]}$ in the fixed locus of $\iota_S^{[2]}$; the other is associated to the rational double cover $Z_S\dashrightarrow Y_Z$. This is done in Section \ref{subsec: the picard group of YS}. 

In order to better describe the varieties constructed, we observe that by our assumptions the group generated by $\sigma\in\Aut(S\times S)$ and by $\iota_S\times \id\in\Aut(S\times S)$ acts on $S\times S$ and is isomorphic to the dihedral group of order 8. In Section \ref{sec: quotients} we describe the quotients of $S\times S$ by subgroups of this group. Let $W$ be the surface $S/\iota_S$ and let $\left(S\times S\right)/\langle \sigma,\iota_S\times\id\rangle$ be $W^{(2)}$ (where $W^{(2)}$ is the quotient of $W\times W$ by the automorphism which switches the two pairs of $W$). So the 4-folds $Y_S$ and $Z_S$ are birational to (possibly singular) 4-folds which are respectively $2:1$ and $2^2:1$ covers of $W^{(2)}$.
The 4-fold $\left(S\times S\right)/\langle \sigma, \iota_S\times \iota_S\rangle$ is by construction birational to $Y_S$ and it is also birational to the blow up, $\widetilde{Z_S/\sigma_Z}$, of $Z_S/\sigma_Z$ in its singular locus. We will prove that the 4-folds $Y_S$ and $\widetilde{Z_S/\sigma_Z}$ are isomorphic (and not only birational) if the involution $\iota_S$ is free on $S$. When $\iota_S$ fixes exactly one curve on $S$,  $\left(S\times S\right)/\langle \sigma, \iota_S\times \iota_S\rangle$ is singular along three surfaces meeting transversally in a curve and the 4-folds $Y_S$ and $\widetilde{Z_S/\sigma_Z}$ are obtained by blowing up these singular surfaces in a different order.

In Section \ref{sec: projective models} we give a detailed geometric descriptions of $Y_S$, $Z_S$ and $S^{[2]}$ and of linear systems on them, under some conditions on $(S,\iota_S)$. Indeed, we first explain how to compute the dimension of certain linear systems induced on these varieties by nef and big divisors on $S$ in Theorem \ref{prop: h^0 for the divisors H}, and then we explicitly describe projective models associated to certain linear systems. In particular, we show that if $S$ is a $2:1$ cover of $\mathbb{P}^2$, $Y_S$ is a $2:1$ cover of $(\mathbb{P}^2)^{(2)}$ embedded in $\mathbb{P}^5$ and $Z_S$ is a $2:1$ cover of $\mathbb{P}^2\times \mathbb{P}^2$ embedded in $\mathbb{P}^8$ by the Segre embedding, see Proposition \ref{prop: maps if NS=2}.

If $S$ admits a genus 1 fibration, then $S^{[2]}$ admits a Lagrangian fibration whose smooth fibers are Abelian surfaces $A_t$ (generically isomorphic to a product of two elliptic curves), $Z_S$  and $Y_S$ admit fibrations whose fibers are the Kummer surfaces of the Abelian surfaces $A_t$. Moreover $Z_S$ has an elliptic fibration and a fibration in Calabi--Yau 3-folds of Borcea--Voisin type, see Propositions \ref{prop: maps if NS=U(2)}, \ref{prop: maps S/iota=P1P1}, \ref{prop: fibrations FX, S/iota=F4}.

\subsection*{Acknowledgements}

The authors would like to thank Prof. Bert van Geemen for interesting discussions and for his precious comments on a preliminary version of this paper, and Prof. Keiji Oguiso for an interesting and useful discussion.

\section{Preliminaries}\label{sec: preliminaries}\label{intro-auto}

In this Section we collect some known results which are useful in the sequel. 
\begin{defi}
Le $Y$ be a smooth compact K\"ahler manifold of dimension $n$. Then $Y$ is called Calabi--Yau variety if 
\begin{itemize}
\item The canonical bundle of $Y$ is trivial and
\item $h^{i,0}(Y)=0$ for every $i=1,\ldots, n-1.$
\end{itemize}
\end{defi}
We underline that we do not require a Calabi--Yau variety to be simply connected. 
\begin{defi}
Let $X$ be a smooth compact K\"ahler manifold. Then $X$ is called \hk variety or, equivalently, IHS variety if 
\begin{itemize}
\item $X$ is simply connected
\item $H^{2,0}(X)=\mathbb{C}\omega_X$, where $\omega_X$ is a symplectic form.
\end{itemize}
\end{defi}
We observe that the existence of a symplectic form on a \hk variety $X$ implies that the canonical bundle of $X$ is trivial and the complex dimension of $X$ is even.

Fundamental examples of \hk  manifolds were discovered by Beauville \cite{beau83}: he produced two families of \hk manifolds in every even dimension and they are constructed as follows. Let $S$ be a $K3$ surface and let $S^{[n]}$ denote the Hilbert scheme of length $n$ zero dimensional subschemes of $S$. Then $S^{[n]}$ is a resolution of the $n$-th symmetric product $S^{(n)}$ and it has a unique symplectic form. K\"ahler deformations of $S^{[n]}$ are called manifolds of $\hskn$ type. Similarly, if $A$ is an Abelian surface, $A^{[n+1]}$ has a symplectic form and a fibre $K_n(A)$ of the Albanese map is \hk and is called a generalized Kummer manifold.

The existence of a symplectic form provides a canonically defined quadratic form on the second cohomology of \hk manifolds, which is given in terms of the top self intersection of divisors. This form is usually called the Beauville--Bogomolov--Fujiki form and gives a lattice structure to the second integral cohomology. In the above two examples, the lattices are the following:
\[ H^2(S^{[n]},\mathbb{Z}) \cong H^2(S,\mathbb{Z})\oplus \mathbb{Z}\delta \cong U^3\oplus E_8(-1)^2\oplus (-2n+2), \]
\[ H^2(K_n(A),\mathbb{Z}) \cong H^2(A,\mathbb{Z})\oplus \mathbb{Z}\delta \cong U^3\oplus (-2n-2),\]
where $2\delta$ is the class of the exceptional divisor of the map $S^{[n]}\rightarrow S^{(n)}$ or of the Albanese fibre of the $A^{[n+1]}\rightarrow A^{(n+1)}$, respectively.\\

In the following we will concentrate ourselves on 4-dimensional \hk varieties so we fix here some useful notation.
By construction, $S^{[2]}$ is the blow up of the singular quotient of $S\times S$ by the map which switches the two factors. 
\begin{defi}
We denote by $\sigma\in\Aut(S\times S)$ the map $\sigma:(p,q)\mapsto (q,p)$, for each $(p,q)\in S\times S$, i.e.\ $\sigma$ is the map that switches the two factors in $S\times S$.

If $g$ is an automorphism of a K3 surface $S$ (resp. an Abelian surface $A$), then $g\times g\in\Aut(S\times S)$ (resp. $g\times g\times g\in\Aut(A\times A\times A)$) induces a unique automorphism on $S^{[2]}$ (resp. $K_2(A)$), called natural automorphism induced by $g$. 

The natural automorphism of $S^{[2]}$ induced by $g$ is denoted by $g^{[2]}$. 

We denote by $\iota_S$ a non symplectic involution on $S$ and thus by $\iota_S^{[2]}$ the natural non symplectic involution induced by $\iota_S$ on $S^{[2]}$. 
\end{defi}

\section{Quotients of hyperk\"ahler varieties by automorphisms}\label{sec: general-teo}

In the following we will consider quotients of \hk varieties by certain finite automorphisms and crepant resolutions of these quotients. 
The main theorem of this section is that there exists a crepant resolution of the quotient of a hyperk\"ahler variety of dimension $2p$ by a non symplectic automorphism of prime order $p$ whenever the fixed locus is pure of dimension $p$; this hypothesis is surely satisfied for non symplectic involutions on hyperk\"ahler $4$-folds.
First we recall some basic definitions and known results.

\begin{defi} 
Let $V$ be smooth variety of dimension $m$ and let $\alpha\in \Aut(V)$ be an automorphism of order $p$ with a non-empty fixed locus. 

We denote by $A$ be the matrix which linearizes $\alpha$ near a component $C$ of its fixed locus and let $(\zeta_p^{a_1},\ldots \zeta_p^{a_{m}})$, with $0\leq a_i<p$, be its eigenvalues. 

The age of $\alpha$ near $C$ is defined as $(\sum_{i=1}^m a_i)/p$.
\end{defi}

\begin{prop}{\rm (see for example \cite[Theorem 6.4.3 and Proposition 6.4.4]{joyce})}\label{prop: terminal sing no crepant}
With the same notation as before, let us assume that $A\in SL(\C,n)$. The quotient $Z/\alpha$ has canonical non terminal singularities on the image of $C$ if and only if the age of $\alpha$ near $C$ is 1.

If $Z/\alpha$ has terminal singularities, then it does not admit a crepant resolution.
\end{prop}

\begin{defi} Let $V$ be a hyperk\"ahler variety of dimension $2n$ and let $\omega_V$ be the symplectic form on $V$; the automorphism $\alpha$ is said to be {\it symplectic} if $\alpha^*\omega_V=\omega_V$, {\it non symplectic} otherwise.\end{defi} 

When $G=\langle\alpha\rangle$ acts non symplectically and is cyclic of prime order $p$, it is always possible to find a generator, that we keep on denoting $\alpha$, such that $\alpha^*\omega_V=\zeta_p\omega_V$ for $\zeta_p$ a primitive $p$-th root of unity.

Let us consider a local system of coordinates $(x_1,\ldots,x_{2n})$ near a connected component of the fixed locus. The automorphism $\alpha$ linearizes  to a matrix $A$ in $GL(\C,2n)$. If the order of $\alpha$ is a prime number $p$, then the eigenvalues of $A$ are $(\zeta_p^{a_1},\ldots \zeta_p^{a_{2n}})$.
 
Recall the following standard fact about the components of the fixed locus.

\begin{lemma}\label{codim-fix}
Let $V$ be a hyperk\"ahler manifold of dimension $2n$ and let $\alpha$ be a non degenerate automorphism of finite order acting on $X$. Then the fixed locus has codimension $\geq 2$ if $\alpha$ is symplectic, and $\geq n$ if $\alpha$ is non symplectic.
\end{lemma}
\proof
When $\alpha$ is symplectic this follows from the fact that every connected component of the fixed locus is symplectic, as shown in \cite[Proposition 3.5]{Oguiso-Kummers}.

When $\alpha$ is non symplectic, suppose on the contrary that there exists a connected component of codimension $<n$. Then the eigenspace relative to the eigenvalue $1$ of $A$ would be an isotropic subspace of dimension $>n$, and this is impossible.
\endproof

\begin{lemma}\label{local-form}
Let $V$ be a hyperk\"ahler variety of dimension $2n$ and let $\alpha\in \Aut(V)$ be an automorphism of order $p$ with a non-empty fixed locus. Let $A$ be the matrix which linearizes $\alpha$ near a component $C$ of its fixed locus and let $(\zeta_p^{a_1},\ldots \zeta_p^{a_{2n}})$, with $0\leq a_i<p$, be its eigenvalues.

Then, there exist local coordinates $(x_1,\dots,x_{2n})$ in an open neighbourhood containing $C$ such that:
\begin{enumerate}
 \item if $\alpha$ is symplectic, the spectrum of $A$ is the union of $n$ pairs of the form $(\zeta_p^{a_j},\zeta_p^{p-a_j})$ with $a_j>0$, for $j=1,\dots,n$;
 \item if $\alpha$ is non symplectic, the spectrum of $A$ is the union of $s\leq n$ pairs of the form $(1,\zeta_p)$ or $(\zeta_p^{a_j},\zeta_p^{p+1-a_j})$ with $a_j>0$, for $j=1,\dots,s$, plus the eigenvalue $\zeta_p^{\frac{p+1}{2}}$ with multiplicity $2n-2s$.
\end{enumerate}

\end{lemma}

\proof

Fix a component of the fixed locus; we choose local coordinates $(x_1,\ldots,x_{2n})$ such that the symplectic form $\omega_V$ is represented by the standard symplectic matrix $J$. 

Since $\alpha$ is an automorphism of $V$, it preserves the Hodge decomposition of $H^2(V,\C)$ and so $\alpha^*(\omega)=\lambda_\alpha \omega$ where $\lambda_{\alpha}\in \C^*$. Moreover, since $\alpha$ has order $p$, $\lambda_{\alpha}^p=1$. If $\alpha$ is a symplectic automorphism, then $\lambda_{\alpha}=1$, otherwise, without loss of generality, we can assume that $\lambda_{\alpha}=\zeta_p$, where $\zeta_p=e^{2i\pi/p}.$

Since $\alpha$ preserves $H^{2,0}(V)$, locally we have $^tAJA=\lambda_{\alpha}J$, hence $A^{-1}=\lambda_{\alpha}^{-1}J^{-1}\ ^tAJ$. Let $\mu=\zeta_p^{a_j}$ be one of the eigenvalues of $A$; then $\lambda_{\alpha}\mu^{-1}$ is an eigenvalue of $A$ as well, distinct from $\mu$ if $\mu^2\neq \lambda_{\alpha}$. Thus we obtain the following possibilities:

\begin{enumerate}
 \item $p=2$: in this case, all the eigenvalues satisfy $\mu^2=1$. If $\alpha$ is symplectic, all the connected components of the fixed locus are symplectic \cite[Proposition 3]{CamereInvolutions}, hence even-dimensional, and this implies that $\pm1$ occur both with even multiplicity, so that  $\det A=1$. If  $\alpha$ is non symplectic, all the connected components are Lagrangian submanifolds \cite[Lemma 1]{BeauvilleInvolutions}, so that the multiplicity of $\mu=1$ is exactly $n$ and $\det A=(-1)^{n}$.
 \item $p>2$, $\lambda_{\alpha}=1$: the eigenvalues of $\alpha$ are $n$ pairs of the form $(\zeta_p^{a_j},\zeta_p^{p-a_j})$ with $a_j>0$, for $j=1,\dots,n$, and the determinant is $\det A=\zeta_p^{\sum_{i=1}^{n} (a_{i}+p-a_{i})}=1$.
 \item $p>2$, $\lambda_{\alpha}=\zeta_p$: the eigenvalues of $\alpha$ are $s\leq n$ pairs of the form $(1,\zeta_p)$ or $(\zeta_p^{a_j},\zeta_p^{p+1-a_j})$ with $a_j>0$, for $j=1,\dots,s$, plus the eigenvalue $\zeta_p^{\frac{p+1}{2}}$ with multiplicity $2n-2s$. Here $\det A=\zeta_p^{\sum_{i=1}^{s} (a_{i}+p-a_{i}+1)+(n-s)(p+1)}=\zeta_p^n$.
\end{enumerate}

\endproof

Singular quotients of hyperk\"ahler manifolds have already been studied in the literature, although the accent has always been on using quotients by symplectic automorphisms to construct singular symplectic manifolds and look for possible desingularizations. The following results include also results previously obtained by Fujiki \cite{Fujiki} and Menet \cite[Proposition 2.39]{MenetPhD}.

\begin{proposition}\label{prop: symplectic case}
 Let $V$ be a hyperk\"ahler variety of dimension $2n$ and let $\alpha\in \Aut(V)$ be a symplectic automorphism of order $p$ with a non-empty fixed locus. Let $A$ be the matrix which linearizes $\alpha$ near a component of its fixed locus and let $(\zeta_p^{a_1},\ldots \zeta_p^{a_{2n}})$, with $0\leq a_i<p$, be its eigenvalues.
 
In this case $\alpha$ preserves the volume form of $V$, and the singularities of $V/\alpha$ are canonical and not terminal if and only if all the components of the fixed locus have dimension $2n-2$.
\end{proposition}

\proof
Let $C$ be a connected component of the fixed locus of $\alpha$ and let $(x_1,\dots,x_{2n})$ be local coordinates as in Lemma \ref{local-form}.

We recall that the volume form $\Omega$ of $V$ is a complex multiple of $\omega_V^n$; we can assume that $$\Omega:=k\omega_V^n= dx_1\wedge dx_2\wedge \ldots\wedge dx_{2n}$$
for a certain constant $k\in\C^*$. This implies that the action of $\alpha^*$ on $\Omega$ is the multiplication by the determinant of $A$, so we have 
$$\alpha^*(\Omega)=\det(A)\Omega.$$

Period preserving automorphisms are exactly those whose linearization $A$ belongs to $SL(\C,2n)$. Since $\alpha$ is symplectic, by Lemma \ref{local-form} $\det(A)=1$ and $\alpha^*(\Omega)=\Omega$.

Let us consider a component $C\subset V$ of the fixed locus of $\alpha$; $C$ has codimension greater than $1$ by Lemma \ref{codim-fix}. The quotient $V/\alpha$ is singular in $\pi(C)\subset V/\alpha$, where $\pi:V\ra V/\alpha$ is the quotient map. By Proposition \ref{prop: terminal sing no crepant}, the singularity $\pi(C)$ is canonical but not terminal if and only if the age of $\alpha$ near $C$ is 1, i.e. if $(\sum_{j=1}^{2n}a_j)/p=1$.

We distinguish two cases:

\begin{enumerate}
 \item If $p=2$, denote by $2k$ the multiplicity of $1$ as an eigenvalue; then the age $(\sum_{j=1}^{2n}a_j)/2=\frac{2n-2k}{2}=n-k$ equals $1$ if and only if $k=n-1$. As a consequence, $C$ has codimension two.
 \item If $p>2$, then $\sum_{j=1}^{2n}a_j=\sum_{i=1}^{n}(a_{i}+p-a_{i})=\sum_{i=1}^n k_i p$. Let us consider the possibilities for the $2\times 2$ block diagonal matrix $A_i:=\diag(\zeta_p^{a_{i}},\zeta_p^{p-a_{i}})$: either it is the identity matrix, and in this case $k_i=0$, or it is not, and in this case $k_i=1$. If we require that $\sum_{j=1}^{2n}a_j/p=1$ we are requiring that exactly one $k_i$ is non zero, which implies that $C$ has codimension 2, for each component $C$ in the fixed locus of $\alpha$. 
\end{enumerate}

Vice versa, if every component $C$ of the fixed locus of $\alpha$ has dimension $2n-2$, then there are exactly two eigenvalues $A$ which are not equal to $1$. Since $\alpha$ is symplectic, they are of the form $(\zeta_p^{a_j},\zeta_p^{p-a_j})$ by Lemma \ref{local-form}. This implies that the singularities of $V/\alpha$ are canonical but not terminal.

\endproof

In \cite[Proposition 2.39]{MenetPhD}, the interested reader can find an explicit list of quotients of this kind and more details on the existence of a resolution of singularities of $V/\alpha$.

\begin{theorem}\label{prop: main-thm higher dim}
Let $V$ be a hyperk\"ahler variety of dimension $2n$ and let $\alpha\in \Aut(V)$ be a non symplectic automorphism of prime order $p$ with a non-empty fixed locus. 
Then:
\begin{enumerate}
 \item $\alpha$ preserves the volume form if and only if $p|n$.
 \item The singularities of $V/\alpha$ are canonical and not terminal if and only if $p=n$ and all the components of the fixed locus of $\alpha$ have dimension $p=n$.
\end{enumerate}

In particular, if $V$ is a $2p$-dimensional hyperk\"ahler variety and $\alpha$ is a non symplectic automorphism of order $p$ of $V$ such that all the components of the fixed locus of $\alpha$ have dimension $p$, then the blow up of $V/\alpha$ in its singular locus is a Calabi--Yau $2p$-fold.
\end{theorem}

\proof 
Let $C$ be a connected component of the fixed locus of $\alpha$ and let $(x_1,\dots,x_{2n})$ be local coordinates as in Lemma \ref{local-form}.

As in proof of Proposition \ref{prop: symplectic case}, we obtain that

$$\alpha^*(\Omega)=\det(A)\Omega.$$

Period preserving automorphisms are exactly those whose linearization $A$ belongs to $SL(\C,2n)$. Lemma \ref{local-form} implies that  $\alpha^*(\Omega)=\zeta_p^n\Omega$ if $\alpha$ is non symplectic. Since $\zeta_p^n=1$ if and only if $p|n$, we obtain that a non symplectic automorphism of order $p$ of $V$ preserves the volume form if and only if $p|n$.

Let us consider a component $C\subset V$ of the fixed locus of $\alpha$; $C$ has codimension greater than $n-1$ by Lemma \ref{codim-fix}. The quotient $V/\alpha$ is singular in $\pi(C)\subset V/\alpha$, where $\pi:V\ra V/\alpha$ is the quotient map. The singularity $\pi(C)$ is canonical but not terminal if and only if $(\sum_{j=1}^{2n}a_j)/p=1$.

We know that $p|n$, so we can write $n$ as $n=n'p$, $n'\in \N_{>0}$. Moreover we know that there are $s$ pairs of distinct eigenvalues $(\zeta_p^{a_j},\zeta_p^{a_{h_j}})$ of the form $(1,\zeta_p)$ or $(\zeta_p^{a_j},\zeta_p^{p+1-a_j})$ with $a_j>0$ and that $\zeta_p^{\frac{p+1}{2}}$ occurs with multiplicity $2n-2s$ . We can assume without loss of generality that $0\leq a_j<p$ for every $j=1,\ldots, 2s$ and that $a_j<a_{h_j}$. So $a_{j}+a_{h_j}=1+k_jp$ with $k_j=0$ if $a_{j}=0$ and $a_{h_j}=1$, and $k_j>0$ otherwise (i.e. if $a_{j}>0$ and so $a_{h_j}>0$). Hence $(\sum_{j=1}^{2n}a_j)/p=(\sum_{i=1}^{s}(a_{j}+a_{h_j}))/p+(n-s)(p+1)/p=\sum_{i=1}^s( 1+k_ip)/p+(n-s)(p+1)/p=n'+\sum_{i=1}^s k_i+n-s$. Now the condition $(\sum_{j=1}^{2n}a_j)/p=1$ implies that $n'=1$, $s=n$ and $k_i=0$ for each $i=1,\ldots, n$. This implies that $n=p$ and $a_{j}=0$ for every $j=1,\ldots, p$. So the eigenvalues of $A$ are $1$ and $\zeta_p$, both with multiplicity $p$. Thus, locally the fixed locus could be described by $x'_{2i}=0$, $i=1,\
ldots,p$, its codimension is $p$ and its dimension is $p$.

Vice versa, if $\alpha$ is a non symplectic automorphism of a $2p$-dimensional hyperk\"ahler variety $V$ such that the fixed locus of $\alpha$ consists of subvarieties of dimension $p$, then each block $A_i$ of the eigenvalues matrix is $A_i:=\diag (\zeta_p^0,\zeta_p^1)$ and it is immediate to check that $\sum_{j=1}^n a_j/p=1$.

Finally, let us show that there exists a crepant resolution $Y$ of $V/\alpha$. Let us blow up $V$ in the fixed locus of $\alpha$. So we are blowing up $V$ in a disjoint union of $p$-dimensional smooth subvarieties, introducing a (not necessarily connected) exceptional divisor $E$. We denote by $\beta:\widetilde{V}\ra V$ this blow up. The automorphism $\alpha\in \Aut(V)$ induces an automorphism $\widetilde{\alpha}$ on $\widetilde{V}$ which acts as the identity on $E$ and acts freely on $\widetilde{V}-E$. So the quotient $\widetilde{V}/\widetilde{\alpha}$ is a smooth variety, which is isomorphic to the blow up of $V/\alpha$ in its singular locus. It remains to prove that the canonical bundle of $\widetilde{V}/\widetilde{\alpha}$ is trivial: first one computes  $$K_{\widetilde{V}}=\beta^*(K_V)+(p-1)E$$
so that $K_{\widetilde{V}}=(p-1)E$.
Then one observes that the quotient map $\pi':\widetilde{V}\ra \widetilde{V}/\widetilde{\alpha}$ exhibits $\widetilde{V}$ as $p:1$ cyclic cover of $\widetilde{V}/\widetilde{\alpha}$ branched on $B:=\pi'(E)$. Hence there exists a divisor $L\in \Pic(\widetilde{V}/\widetilde{\alpha})$ such that $pL\simeq B$ and $K_{\widetilde{V}}=\pi'^*(K_{\widetilde{V}/\widetilde{\alpha}}+(p-1) L)$. Multiplying both terms by $p$ one obtains  $$pK_{\widetilde{V}}=\pi'^*(pK_{\widetilde{V}/\widetilde{\alpha}})+\pi'^*((p-1) pL).$$ Recalling that $pL\simeq B$ and $\pi'^*(B)=pE$ one has 
$$p(p-1)E=\pi'^*(pK_{\widetilde{V}/\widetilde{\alpha}})+p(p-1)E,$$
which implies that $\pi'^*(pK_{\widetilde{V}/\widetilde{\alpha}})$ is trivial. 
On the other hand $\widetilde{V}/\widetilde{\alpha}$ is the blow up of $V/\alpha$ in its singular locus and $V/\alpha$ has trivial canonical bundle (because it is the quotient of $V$, which has trivial canonical bundle, by a period preserving automorphism). Let us denote by $\beta':\widetilde{V}/\widetilde{\alpha}\ra V/\alpha$ this blow up. We obtain $K_{\widetilde{V}/\widetilde{\alpha}}=\beta'^*(K_{V/\alpha})+hB=hB$ for a certain $h\in \Q$. Since  $\pi'^*(pK_{\widetilde{V}/\widetilde{\alpha}})$ is trivial one has that $\pi'^*(phB)=p^2hE$ is trivial. The divisor $E$ is effective, so $p^2hE=0$ implies $h=0$.   

Moreover, $h^{i,0}(\widetilde{V})=h^{i,0}(V)$ because they are birational invariants; and on the other hand, $h^{i,0}(\widetilde{V}/\widetilde{\alpha})=\dim H^{i,0}(\widetilde{V})^{\widetilde{\alpha}}=0$ for $0<i<2p$, since $\alpha$ does not preserve the symplectic structure of $V$.

But a smooth variety $Y$ of dimension $2p$ with trivial canonical bundle and $h^{i,0}(Y)=0$ for $0<i<2p$ is Calabi--Yau.\endproof

If $V$ is a hyperk\"ahler $2n$-fold and $\alpha$ is a non symplectic involution on $V$, then the components of the fixed locus of $\alpha$ are Lagrangian submanifolds of $V$ and thus in particular they have dimension $n$ \cite[Lemma 1]{BeauvilleInvolutions}. So a non symplectic involution on $V$ preserves the canonical bundle and is such that $V/\alpha$ has canonical singularities if and only if $n=2$. This will be the setting of the rest of the paper.

\begin{rem}{\rm To the best of our knowledge there are no examples of pairs $(V,\alpha_V)$ which satisfy Theorem \ref{prop: main-thm higher dim} such that $\dim V>4$.}\end{rem}

\section{Quotients of hyperk\"ahler 4-folds by non symplectic involutions}\label{sec: quotients of hyperkahler}

From now on, $V$ is a hyperk\"ahler 4-fold and $\alpha$ is a non symplectic involution (so $n=p=2$).
Hence $V/\alpha$ admits a crepant resolution which is a Calabi--Yau 4-fold.
The aim of this Section is to describe an explicit crepant resolution of $V/\alpha$ (Section \ref{subsec: crepant resoltuion}), which allows to construct a Calabi--Yau 4-fold and to compute its Hodge diamond starting from some information on the action of $\alpha$ on $V$ (Theorem \ref{prop: hodge diamonds in general}). Then we apply these results to some specific \hk 4-folds with a non symplectic involution, in particular we will consider quotients of generalized Kummer in Section \ref{subsec: generalized Kummer} and K3 type \hk 4-folds in Section \ref{subsect: HK defomration equiv to Hilb2}.
 
\subsection{The crepant resolution}\label{subsec: crepant resoltuion}
Let us denote by $\pi:V\ra V/\alpha$ the quotient map and by $B_j$ the irreducible components of the fixed locus of $\alpha$. We recall that $B_j$ are smooth surfaces.  In order to construct the crepant resolution $\widetilde{V/\alpha}$ of $V/\alpha$ considered in proof of Theorem \ref{prop: main-thm higher dim} it suffices to blow up once the singular locus of $V/\alpha$, which is given by $\coprod_j \pi(B_j)$. This introduces one divisor for each component $B_j$ of the fixed locus, and this divisor is a $\mathbb{P}^1$-bundle over $B_j$.

There exists the following commutative diagram:
$$
\xymatrix{ V\ar[d]_{\pi}&\widetilde{V}\ar[l]_{\beta}\ar[d]_{\pi'}\\
V/\alpha&\widetilde{V}/\widetilde{\alpha}\ar[l]_{\beta'}},
$$ 
where $\beta$ is the blow up of $V$ in the components of the fixed locus of $\alpha$, $\widetilde{\alpha}$ is the automorphism of $\widetilde{V}$ induced by $\alpha$, $\pi$ and $\pi'$ are the quotient maps and $\beta'$ is the blow up of $V/\alpha$ in its singular locus (equivalently $\beta'$ is the map induced by $\beta$ on the quotient 4-folds). So $\widetilde{V/\alpha}$ is isomorphic to $\widetilde{V}/\widetilde{\alpha}$.

\subsection{The computation of the Hodge numbers}\label{subsec: Computation of the Hodge numbers}
Let $\alpha$ be a non symplectic involution of $V$ and $Y:=\widetilde{V/\alpha}$ the crepant resolution of $V/\alpha$ described in Section \ref{subsec: crepant  resoltuion}.
The 4-fold $Y$ is a Calabi--Yau variety and so its Hodge diamond is invariant under birational transformation (see \cite[Theorem 1.1]{Batyrev-birat-betti} and \cite{Denef-Loeser}). In particular, we can deduce the Hodge diamond of $Y$ from the computation of the $\widetilde{\alpha}$ invariant part of the cohomology of $\widetilde{V}$. Equivalently, one can use the well known formula of the orbifold cohomology. The result, which is not surprising, is that the Hodge diamond of $Y$ depends on the $\alpha$ invariant part of the Hodge decomposition of $V$, and on the Hodge diamond of the fixed locus. After fixing some notation, we summarize the result in Theorem \ref{prop: hodge diamonds in general}.

Let $\coprod B_j$ be the fixed locus of $\alpha$. Let us denote by  $b:=h^0(\coprod B_j)$, i.e. $b$ is the number of components of the fixed locus of $\alpha$, by $c:=\sum_{j=1}^b(h^{1,0}(B_j))$, $d:=\sum_{j=1}^b(h^{2,0}(B_j))$, $e:=\sum_{j=1}^b(h^{1,1}(B_j))$.
Moreover, let us denote by $$t_{1,1}:=\dim(H^2(V,\C))^{\alpha}=\dim (H^{1,1}(V))^{\alpha},$$ where the last equality follows from the fact that $\alpha$ is non symplectic. We set also: $$t_{2,1}:=\dim(H^{2,1}(V)^{\alpha})=\dim(H^3(V,\C)^\alpha),\ \  t_{3,1}:=\dim(H^{3,1}(V)^{\alpha}),\ \ t_{2,2}:=\dim(H^{2,2}(V)^{\alpha}).$$ Since $H^{4,0}(V)$ and $H^{0,4}(V)$ are invariant for $\alpha$, we have the following relation: $\dim(H^4(V,\C)^{\alpha})=2+2t_{3,1}+t_{2,2}$.
\begin{theorem}\label{prop: hodge diamonds in general}
The Hodge diamond of any Calabi--Yau birational to $Y$ is given by 
$$
\begin{array}{ll}
h^{0,0}(Y)=h^{4,0}=1,&
h^{1,0}(Y)=h^{2,0}=h^{3,0}=0,\\
h^{1,1}(Y)=t_{1,1}+b,&
h^{2,1}(Y)=t_{2,1}+c,\\
h^{2,2}(Y)=t_{2,2}+e,&
h^{3,1}(Y)=t_{3,1}+d.\\
\end{array}
$$

\end{theorem}

\proof
The crepant resolution $Y$ of $V/\alpha$ is isomorphic to $\widetilde{V}/\widetilde{\alpha}$ by Section \ref{subsec: crepant resoltuion} and any over Calabi--Yau birational to $Y$ has the same Hodge numbers as $Y$, by \cite[Theorem 1.1]{Batyrev-birat-betti}. The statement now thus follows from the fact that $H^{*,*}(Y)=H^{*,*}(\widetilde{V})^{\widetilde{\alpha}}$. Indeed by standard results on the cohomology of blow ups (see \cite[Theorem 7.31]{Vo}, $h^{p,q}(\widetilde{V})=h^{p,q}(V)+h^{p-1,q-1}(\coprod B_j)$. Since the exceptional divisors of the blow up $\beta:\widetilde{V}\ra V$ are preserved by $\widetilde{\alpha}$, we obtain $\dim\left( H^{p,q}(\widetilde{V})^{\widetilde{\alpha}}\right)=\dim\left( H^{p,q}(V)^\alpha\right)+h^{p-1,q-1}(\coprod B_j)$. 
\endproof
\subsection{Computations in special cases}
The aim of this section is to apply Theorem \ref{prop: hodge diamonds in general} to special hyperk\"ahler 4-folds $V$ with a given non symplectic involution $\alpha$ such that either there exist some relations among the numbers $t_{i,j}$, $b$, $c$, $d$, $e$, or some of these numbers are determined.

\subsubsection{Generalized Kummer four-folds}\label{subsec: generalized Kummer}
Non symplectic involutions on generalised Kummer fourfolds have been recently classified in \cite{mtwkum}. The cohomology of generalised Kummer fourfolds has been studied in detail by Hassett and Tschinkel \cite{HT_kum} and Oguiso \cite{Oguiso-Kummers}, let us review the results relevant for our purposes.\\ 
The Hodge diamond of the generalized Kummer fourfolds is as follows:

\begin{eqnarray*}\begin{array}{lllllll}
1&&&&\\
&0&&&\\
5&&1&&\\
&4&&0&\\
96&& 5 &&1.\\
\end{array}\end{eqnarray*}

Here, the third cohomology has trivial $H^{3,0}$ and four dimensional $H^{2,1}$ which, in the following special setting, can be constructed explicitly. Let $A$ be an abelian surface and let $\mathbb{C}^2$ be its universal cover, with coordinates $z$ and $w$. Let $A(2)$ be the subset of $A^3$ where all points sum to zero. The universal cover of $A(2)$ is the closed submanifold of $(\mathbb{C}^2)^3$ cut out by the following equations:
\begin{eqnarray*}
z_1+z_2+z_3& = & 0,\\
w_1+w_2+w_3& = & 0.
\end{eqnarray*}
The natural one forms $dz_i$, $dw_i$ on $A(2)$ satisfy the same equations and $dz_1,dz_2,dw_1,dw_2$ form a basis of one forms for $A(2)$. We have the following lemma 
\begin{lemma}\cite[Lem. 3.3]{Oguiso-Kummers}
All $\mathfrak{S}_3$ invariant $(2,1)$-forms of $A(2)$ give nonzero forms of $K_2(A)$.
\end{lemma}
\begin{proof}
Let $\tau$ be a $\mathfrak{S}_3$ invariant $(2,1)$-form. It then descends to a $(2,1)$ form $\overline{\tau}$ of $(A(2))^{(3)}$. Let us consider its resolution $\nu\,:\,K_2(A)\rightarrow (A(2))^{(3)}$. As the Hodge structure of $(A(2))^{(3)}$ is pure, we have a well defined map $\nu^*$ of $(2,1)$ forms. This map is furthermore injective, as is the quotient map $A(2)\rightarrow (A(2))^{(3)}$ on symmetric $(2,1)$ forms, therefore our claim holds.
\end{proof}
A basis for this cohomology is given by averaging forms in the orbit of the four following elements:
\begin{eqnarray*}
dz_1\wedge dw_1\wedge d\overline{z}_1,\ \ 
dz_1\wedge dw_2\wedge d\overline{z}_3,\ \ 
dw_1\wedge dz_1\wedge d\overline{w}_1 \mbox{ and } 
dw_1\wedge dz_2\wedge d\overline{w}_3.
\end{eqnarray*}
This will allow us to compute $t_{2,1}$ on natural automorphism only from the action of the automorphism on $H^1(A)$.
The last tool we need is a determination of $H^4(K_2(A))/\Sym^2H^2(K_2(A))$. This has the following geometric characterisation, where we denote by $A[3]$ the group of the 3-torsion points of $A$:
\begin{lemma}\cite[Sect. 4]{HT_kum}
There is an injective map $A[3]\otimes \mathbb{C}\hookrightarrow H^4(K_2(A))$. Let $W=\sum_{p\in A[3]}p\otimes 1$. Then $W^\perp\subset A[3]\otimes \mathbb{C}$ surjects on $H^4(K_2(A))/\Sym^2H^2(K_2(A))$.
\end{lemma}
We omit this proof, however the embedding is geometric, as the classes $p\otimes 1$ are given by the surfaces of subschemes of $A$ whose support is the order three point $p$, and the sum of these classes lies in the subspace generated by the exceptional divisor and the second Chern class. 

As it is classically known (\cite{FUJ}), there are three families of non symplectic involutions on abelian surfaces. The first family is given in terms of the product $E\times E'$ of two elliptic curves and the involution acts as $-1$ on one curve and as $1$ on the other. The elements of the second family are quotients of elements of the first family by an order two point $t$ such that both projections of $t$ on the elliptic curves $E$ and $E'$ are nontrivial. The third family is an iteration of this procedure with a further quotient by a point of order 2. Let $X_1=K_2(E\times E')$, $X_2=K_2(E\times E'/t)$ and $X_3=K_2(E\times E'/\langle t,t'\rangle)$ and let $\varphi_1,\varphi_2$ and $\varphi_3$ be the three involutions induced on them. In all three cases, we have $t_{1,1}=3$ and $t_{2,1}=2$. The three actions give also the same action on order three points of $A$, with $36$ non trivial orbits and $9$ fixed points. Therefore we have $t_{2,2}=54$, where 10 elements come from the symmetric power of $H^2$ and 
45 arise from the action on order three points (we remind that there is a one dimensional intersection). The fixed locus on $X_1$ is given by points of the form $\{(p,e);(q,f);(p+q,-e-f)\}$, where $p,q\in E[2]$ and $e,f\in E'$ or points of the form $\{(0,-2e),(a,e),(-a,e)\}$, where $a\in E$ and $e\in E'$. These cover six surfaces, one is a $\mathbb{P}^2$ (the fibre of the Albanese map for $(E')^{(3)}$, given by fixed subschemes of the form $\{(0,e),(0,f),(0,-e-f)\}_{e,f\in E'}$), three are $(E')^{(2)}$ (which are $\{(0,e),(p,f),(p,-e-f)\}_{e,f\in E',\,p\in E[2]-0}$), one is $E'\times E'$ (given as $\{(p,e),(q,f),(p+q,-e-f)\}_{e,f\in E',\,\{0,p,q,p+q\}=E[2]}$) and the last is given by the invariant subschemes whose points are not fixed, and this is isomorphic to $E\times\mathbb{P}^1$ blown up along nine points, which are subschemes supported entirely on a triple point of $A$. Therefore we get the following Hodge diamond for a crepant resolution $Y_1$ of $X_1/\varphi_1$: 

\begin{eqnarray*}\begin{array}{lllllll}
1&&&&\\
&0&&&\\
9&&0&&\\
&8&&0&\\
75&& 5 &&1.\\
\end{array}\end{eqnarray*}

In case $X_2=K_2(E\times E'/t)$, the quotient of $E\times E'/t$ of $E\times E'$ identifies some of the components of the fixed locus, and we are left with only three surfaces. The first one is again $\mathbb{P}^2$. The second one is $(E')^{(2)}$ and the last one is the surface of invariant non fixed subschemes, isomorphic to that of the previous case. Therefore the Hodge diamond of a crepant resolution $Y_2$ of $X_2/\varphi_2$ is as follows:

\begin{eqnarray*}\begin{array}{lllllll}
1&&&&\\
&0&&&\\
6&&0&&\\
&4&&0&\\
68&& 4 &&1.\\
\end{array}\end{eqnarray*}
Also on $X_3$ the fixed locus gets smaller, and only $\mathbb{P}^2$ and the surface of invariant non fixed subschemes are left. Thus, the Hodge diamond of a crepant resolution $Y_3$ of $X_3/\varphi_3$ is as follows:
\begin{eqnarray*}\begin{array}{lllllll}
1&&&&\\
&0&&&\\
5&&0&&\\
&3&&0&\\
66&& 4 &&1.\\
\end{array}\end{eqnarray*}

\subsubsection{Four-folds of $K3$ type}\label{subsect: HK defomration equiv to Hilb2}

We want to apply the results of Theorem \ref{prop: hodge diamonds in general} if $V$ is deformation equivalent to the Hilbert scheme of two points of a K3 surface. In this case the Hodge diamond of $V$ is known to be: 
\begin{eqnarray*}\label{Hodge X}\begin{array}{lllllll}
1&&&&\\
&0&&&\\
21&&1&&\\
&0&&0&\\
232&&21&&1.\\
\end{array}\end{eqnarray*}
One of the most relevant properties about the cohomology of $V$ if $V$ is of $K3$ type, is that the cohomology group $H^3(V,\C)$ is trivial and the cohomology group $H^4(V,\C)$ is completely determined by $H^2(V,\C)$. In particular, this implies that $t_{2,1}=0$ and that we can get $t_{3,1}$ and $t_{2,2}$ from $t_{1,1}$.
Indeed, there is an isomorphism of Hodge structures $H^4(V,\C)\cong \Sym^2(H^2(V,\C))$ (see \cite{Guan}). Let us denote by $H^{1,1}(V,\C)_{-1}$ the subgroup of $H^{1,1}(V,\C)$ which is anti invariant for $\alpha$. Then
$$t_{3,1}=\dim(H^{3,1}(V,\C)^{\alpha})= 
\dim(H^{2,0}(V)\otimes H^{1,1}(V)_{-1})= h^{1,1}-t_{1,1} =21-t_{1,1}$$
and 
\begin{eqnarray*}\begin{array}{l}t_{2,2}=\dim (H^{2,2}(V)^{\alpha})=
\dim \left(H^{2,0}(V)\otimes H^{0,2}(V)\oplus \Sym^2\left(H^{1,1}(V)\right)\right)^{\alpha}=\\
\dim \left(H^{2,0}(V)_{-1}\otimes H^{0,2}(V)_{-1}\oplus\Sym^2(H^{1,1}(V)^{\alpha})\oplus\Sym^2(H^{1,1}(V)_{-1})\right)=\\
1+\binom{t_{1,1}+1}{2}+\binom{22-t_{1,1}}{2}=232+t_{1,1}^2-21t_{1,1}.\end{array}\end{eqnarray*}

By considering the invariant and the anti invariant parts of the cohomology of $V$, we obtain two sub-Hodge structures of weight $k$ of $H^k(V,\Q)$.
In particular, we obtain the following two invariant and anti invariant Hodge diamonds:
\begin{align*}\label{Hodge inv anti inv}
\begin{array}{l|r}
\begin{array}{llllll}
1&&&&\\
&0&&&\\
t_{1,1}&&0&&\\
&0&&0&\\
232+t_{1,1}^2-21t_{1,1}&&21-t_{1,1}&&1\\
\end{array}&
\begin{array}{lllll}
0&&&&\\
&0&&&\\
21-t_{1,1}&&1&&\\
&0&&0&\\
-t^2_{1,1}+21 t_{1,1}&&t_{1,1}&&0
\end{array}\end{array}
\end{align*}

This allows an easy computation of the trace of $\alpha$ on $H^*(V,\C)$ and, by means of the topological trace formula, Beauville in \cite{BeauvilleInvolutions} deduces information on the Hodge diamond of the fixed locus of $\alpha$ (in \cite{BeauvilleInvolutions} the trace of $\alpha$ on $H^{1,1}(V)$ is denoted by $t$). The relation between $t$ and $t_{1,1}$ is clearly $t=t_{1,1}-(21-t_{1,1})$, so $t=2t_{1,1}-21$.
By \cite[Theorem 2]{BeauvilleInvolutions}, once denoted by $F=\coprod B_j$ the fixed locus of $\alpha$, we have $$\chi(F)=(t_{1,1}^2-21t_{1,1}+112)/2,\ \ e(F)=2t_{1,1}^2-42t_{1,1}+232.$$
In particular, with the notations of Theorem \ref{prop: hodge diamonds in general}, we have 
\begin{eqnarray*}\begin{array}{l}b-c+d=(t_{1,1}^2-21t_{1,1}+112)/2,\\
2b-4c+2d+e=2t_{1,1}^2-42t_{1,1}+232,\end{array}\end{eqnarray*}
from which it follows
\begin{eqnarray*}\begin{array}{l}b=(112-21t_{1,1}+2c-2d+t_{1,1}^2)/2,\\
e=120-21t_{1,1}+2c+t_{1,1}^2.\end{array}\end{eqnarray*}

\begin{prop}\label{prop: Hodge diamond Y for K3 type, also non natural}
Let $V$ be a \hk 4-fold of K3 type admitting a non symplectic involution $\alpha$ such that the fixed locus of $\alpha$ consists of the disjoint union of a finite number of surfaces $B_j$. As before, denote by: $$t_{1,1}:=\dim(H^{1,1}(V)^{\alpha}),\ c:=\sum_j(H^{1,0}(B_j)),\ d:=\sum_j(H^{2,0}(B_j)),$$ any crepant resolution of $V/\alpha$ is a Calabi--Yau variety with Hodge numbers: 
$$
\begin{array}{ll}
h^{0,0}=h^{4,0}=1,&
h^{1,0}=h^{2,0}=h^{3,0}=0,\\
h^{1,1}=(112-19t_{1,1}+2c-2d+t_{1,1}^2)/2,&
h^{2,1}=c,\\
h^{2,2}=352+2t_{1,1}^2-42t_{1,1}+2c,&
h^{3,1}=21-t_{1,1}+d.\\
\end{array}
$$
\end{prop}

\subsubsection{Beauville's non natural involution}\label{subsec Beauville non natural}

Let $H$ be a double EPW sextic and $\iota_H$ be the cover involution, see \cite{ogr06} for the definition. It is well known that $H$ is deformation equivalent to $S^{[2]}$, where $S$ is a quartic in $\mathbb{P}^3$. In the most generic case (for example if $\Pic(S)=\Z A$, where $A$ is a polarization of degree 4), the fixed locus of $\iota_H$ on $H$ consists of a smooth surface parametrizing the bitangents to $S$ (see \cite{Ferretti}) and $\iota_H$ fixes only one class in $H^{1,1}(H)$, so $t_{1,1}=1$.

Let $X'$ be the double cover of $\mathbb{P}^3$ branched along a generic quartic $W\subset\mathbb{P}^3$. Let $F'$ be the set of the conics on $X'$,  and $F_0'$ be the parametrizing space of the bitangents to $W$. Both $F'$ and $F_0'$ are surfaces and there exists a non ramified double cover $F'\ra F_0'$. In \cite[p. 533]{Logachev} it is proved that the Hodge numbers of $F'$ and $F_0'$ are the same as the ones of analogous surfaces defined for a different threefold $X$ (a Fano threefold of genus 6). This allows to compute them:
$$h^{0,0}(F_0')=1,\ h^{1,0}(F_0')=0,\ \ h^{2,0}(F_0')=45,\ \ h^{1,1}(F_0')=100$$
(cf. \ \cite[(0.7)]{Logachev}, see also \cite[(0.6)]{Logachev} for the Hodge numbers of $F'$).

 The fixed locus $\Fix_{\iota_H}(H)$ is the surface $F_0'$ considered before. This allows to compute the Hodge numbers of the crepant resolution $Y$ of $H/\iota_H$ (as in the previous section):
$$
\begin{array}{l}
h^{0,0}(Y)=1,\ \ h^{1,0}(Y)=h^{2,0}=h^{3,0}=0,\\
h^{1,1}(Y)=\dim(H^{1,1}(H)^{\iota_H})+\dim(H^{0,0}(F_0'))=1+1=2,\\
h^{2,1}(Y)=\dim(H^{2,1}(H)^{\iota_H})+\dim(H^{1,0}(F_0'))=0+0=0,\\
h^{3,1}(Y)=\dim(H^{3,1}(H)^{\iota_H})+\dim(H^{2,0}(F_0'))=20+45=65,\\
h^{2,2}(Y)=\dim(H^{2,2}(H)^{\iota_H})+\dim(H^{1,1}(F_0'))=212+100=312.\\
\end{array}$$

\section{Non symplectic natural involutions on $S^{[2]}$ and the Calabi--Yau 4-folds $Y_S$ and $Z_S$}\label{sec: natural symplectic involution}

Here we consider the case of natural non symplectic involutions on the Hilbert scheme of two points $S^{[2]}$ of a K3 surface $S$. As already discussed in \S \ref{intro-auto}, if the K3 surface $S$ admits a non symplectic involution $\iota_S$, this induces a non symplectic involution, denoted by $\iota_S^{[2]}$, on $S^{[2]}$. Our first goal will be the construction of the Calabi--Yau variety $Y_S$, the crepant resolution of $S^{[2]}/\iota_S^{[2]}$, and the computation of its Hodge numbers, in Theorem \ref{prop: HOdge diamond natural}.
Then, in Section \ref{subsec: ZS} we recall the construction of other Calabi--Yau 4-folds obtained applying the Borcea--Voisin construction to $S\times S$: this Calabi--Yau will be denoted by $Z_S$ and is the crepant resolution of $\left(S\times S\right)/\langle \iota_S\times \iota_S\rangle$. This allows us to show that there exists an involution on $Z_S$, denoted by $\sigma_Z$, such that $Z_S/\sigma_Z$ is birational to $Y_S$, see Proposition \ref{prop: ZS/sigma birat to YS}.
The existence of this involution depends on the fact that the group $(\Z/2\Z)^2=\langle \iota_S\times \iota_S,\sigma\rangle$ acts on $S\times S$ and the 4-fold $Y_S$ is birational to the quotient of $S\times S$ for the full group, while $S^{[2]}$ and $Z_S$ are birational to the quotients of $S\times S$ by two specific subgroups of order 2 of $(\Z/2\Z)^2$. In Proposition \ref{prop: other hk quotient of SxS}, we will prove that the quotient of $S\times S$ by the third cyclic subgroup of $(\Z/2\Z)^2$ admits a crepant resolution $V$ which is isomorphic to $S^{[2]}$. In Section \ref{subsec: the picard group of YS}, we describe explicitly the Picard group of $Y_S$  and the branch divisors associated to the three double covers $Z_S\dashrightarrow Y_S$, $S^{[2]}\dashrightarrow Y_S$ and $V\dashrightarrow Y_S$.

\subsection{The Calabi--Yau 4-folds $Y_S$}\label{subsec: Y_S}

We first recall the main results on non symplectic involutions on $K3$ surfaces, due to Nikulin, see \cite{Nikulin}.

Let $S$ be a K3 surface admitting a non symplectic involution $\iota_S$. 

We denote by $r:=\mathrm{rank}( H^{2}(S,\Z)^{\iota_S})$ and by $a$ the integer defined by $$(H^2(S,\Z)^{\iota_S})^{\vee}/(H^2(S,\Z)^{\iota_S})\simeq (\Z/2\Z)^a.$$

The fixed locus of $\iota_S$ consists of $N$ disjoint curves. If $N=0$, then $r=a=10$. Otherwise, let us consider the fixed curve with highest genus and denote it by $C$. If the fixed locus of $\iota_S$ consists of two curves of genus one, and in this case $r=10$, $a=8$. If the fixed locus of $\iota_S$ is neither empty nor the disjoint union of two curves of genus one, then $\Fix_{\iota_S}=C\coprod_{i=1}^{N-1} R_i$ where $R_i$ is a rational curve for all $i=0,\dots,N-1$. Let us denote by $g:=g(C)$, by $k:=N-1$, by $N'=\sum_{j=1}^Ng(D_j)$, where the $D_j$ are all the curves in the fixed locus of $\iota_S$.\\
The invariants just introduced satisfy the following relations:
\begin{equation}\label{eq: relations}\begin{array}{ll} k=\left(r-a\right)/2,&g=\left(22-r-a\right)/2,\\
N=\left(2+r-a\right)/2,&N'=\left(22-r-a\right)/2,\\
r=11+k-g=10+N-N',&a=11-k-g=12-N-N'.\end{array}
\end{equation}

Suppose now that the fixed locus of $\iota_S$ on $S$ consists of a curve $C$ of genus $g:=g(C)$ and of $k$ other rational curves.
An application of \cite[\S 4.2]{Boissiere-Natural} shows that the fixed locus of $\iota_S^{[2]}$ on $S^{[2]}$ consists of the following surfaces:
\begin{itemize}
\item one surface isomorphic to $C^{[2]}$ whose Hodge numbers are $h^{0,0}=1$, $h^{1,0}=g$, $h^{2,0}=g(g-1)/2$, $h^{1,1}=1+g^2$;
\item $k$ surfaces isomorphic to $C\times R_j\simeq C\times \mathbb{P}^1$, $j=1,\ldots k$, whose Hodge numbers are $h^{0,0}=1$, $h^{1,0}=g$, $h^{2,0}=0$, $h^{1,1}=2$; 
\item $k$ surfaces isomorphic to $(\mathbb{P}^1)^{[2]}$, whose Hodge numbers are $h^{0,0}=1$, $h^{1,0}=0$, $h^{2,0}=0$, $h^{1,1}=1$;\item $k(k-1)/2$ surfaces isomorphic to $\mathbb{P}^1\times\mathbb{P}^1$, whose Hodge numbers are $h^{0,0}=1$, $h^{1,0}=0$, $h^{2,0}=0$, $h^{1,1}=2$; 
\item one surface which is isomorphic to the smooth quotient surface $S/\iota_S$ hence $h^{0,0}=1$, $h^{1,0}=0$, $h^{2,0}=0$ and $h^{1,1}=r$. 
\end{itemize}
So the Hodge diamond of the fixed locus is: 
$$
\begin{array}{llll}
1+k+k+k(k-1)/2+1&&\\
&g+kg&\\
1+g^2+3k+k(k-1)+r&&g(g-1)/2.
\end{array}
$$

We observe that the topological characteristic of the fixed locus is  $2(r^2-19r+96)$ (where we used the relations \eqref{eq: relations}).\\

Now we consider the action of $\iota_S^{[2]}$ on the cohomology of $S^{[2]}$. Obviously, $\iota_S^{[2]}$ is a specific choice of an involution on a specific hyperk\"ahler variety which is deformation equivalent to (and in fact coincides with) a Hilbert scheme of two points on a K3 surface. So we are exactly in the setting of Section  \ref{subsect: HK  defomration equiv to Hilb2}. Since $\iota_S^{[2]}$ is induced by $\iota_S$ and preserves the exceptional divisor in $S^{[2]}$, we have that $t_{1,1}=r+1$.

Hence, applying Proposition \ref{prop: Hodge diamond Y for K3 type, also non natural}, we obtain the Hodge diamond of any crepant resolution $Y_S$ of $S^{[2]}/\iota_S^{[2]}$:

\begin{theorem}\label{prop: HOdge diamond natural} Let $S$ be a K3 surface and $\iota_S$ be a non symplectic involution on $S$, whose fixed locus is associated to the values $(N,N')$. Let $Y_S$ be the blow up of $S^{[2]}/\iota_S^{[2]}$ in its fixed locus. Then the Hodge numbers of any Calabi--Yau 4-fold birational to $Y_S$ are: $h^{i,0}=0,\ i=1,2,3,\ h^{0,0}=h^{4,0}=1$ and

$$
\begin{array}{llll}
h^{1,1}=&(24+3N-2N'+N^2)/2,\\
h^{2,1}=&NN',\\
h^{3,1}=&(20-2N+N'+N'^2)/2,\\
h^{2,2}=&132+2N-2N'+2N^2-2NN'+2N'^2.\\
\end{array}
$$
\end{theorem}

\begin{rem} Since there are $64$ possible choices for the pairs $(N,N')$ associated to $\iota_S$ we obtain $64$ different Hodge diamonds for the Calabi--Yau 4-folds $Y_S$. In all the admissible pairs, $1\leq N\leq 10$ and $0\leq N'\leq 10$. The complete list of the Hodge diamonds of the Calabi--Yaus $Y_S$ is given in the Appendix \ref{Appendix}.

\end{rem}
\subsection{The Calabi--Yau of Borcea--Voisin type and $Z_S$}\label{subsec: ZS}
Let $B_1$ and $B_2$ be two Calabi--Yau varieties of dimension $n_1$ and $n_2$ respectively. Let us assume that $B_i$ admits an involution $\iota_i$ which does not preserve the period of $B_i$. Let us consider the involution $\iota_1\times \iota_2$ on $B_1\times B_2$. If all the components of the fixed locus of $\iota_1\times \iota_2$ on $B_1\times B_2$ have codimension 2, then there exists a crepant resolution of $(B_1\times B_2)/(\iota_1\times \iota_2)$ which is a Calabi--Yau of dimension $(n_1+n_2)$. The Calabi--Yau manifolds constructed in this way are called of Borcea--Voisin type, after the original independent papers by Borcea \cite{Borcea} and Voisin \cite{Voisin}; the generalization that we have just reviewed is a result by Cynk and Hulek (see \cite[Proposition 2.1]{Cynk-Hulek}), where the authors refer to the construction as the ``Kummer construction''.

In our setting it is quite natural to consider a Calabi--Yau 4-fold of Borcea--Voisin type, by choosing $B_1=B_2=S$ and $\iota_1=\iota_2=\iota_S$, with the same notation of the previous section. The quotient $(S\times S)/(\iota_S\times \iota_S)$ is singular along some surfaces and the blow up of these surfaces gives a smooth Calabi--Yau 4-fold, that we denote by $Z_S$. In this section we compute the Hodge numbers of $Z_S$, which do not depend on the birational model of $Z_S$ which we choose; these computations are not new and can be found in \cite{Dillies}.

Let $S$ be a K3 surface and $\iota_S$ be a non symplectic involution of $S$ whose fixed locus consists of a curve of genus $g$ and $k$ rational curves. 

The Hodge diamond of $S\times S$ is
\begin{eqnarray*}\label{Hodge S2}\begin{array}{ccccccccccc}
1&&&&\\
&0&&&\\
40&&2&&\\
&0&&0&\\
404&&40&&1.\\
\end{array}\end{eqnarray*}

First we consider the action of $(\iota_S\times\iota_S)^*$ on $H^*(S\times S,\Q)$. The invariant and the anti invariant part of $H^*(S\times S,\Q)$ give two sub-Hodge structures of $H^*(S\times S,\Q)$ which are respectively:
\begin{align*}\label{Hodge inv anti inv BV}
\begin{array}{l|r}
\begin{array}{llllllll}
1&&&&\\
&0&&&\\
2r&&0&&\\
&0&&0&\\
2r^2-40r+404&&40-2r&&1,\\
\end{array}&
\begin{array}{lllllllllll}
0&&&&\\
&0&&&\\
40-2r&&2&&\\
&0&&0&\\
40r-2r^2&&2r&&0.
\end{array}\end{array}
\end{align*}

The fixed locus of $\iota_S\times \iota_S$ on $S\times S$ consists of:
\begin{itemize}
\item one surface isomorphic to $C\times C$ whose Hodge numbers are $h^{0,0}=1$, $h^{1,0}=2g$, $h^{2,0}=g^2$, $h^{1,1}=2+2g^2$;
\item $2k$ surfaces isomorphic to $C\times \mathbb{P}^1$ whose Hodge numbers are $h^{0,0}=1$, $h^{1,0}=g$, $h^{2,0}=0$, $h^{1,1}=2$; 
\item $k^2$ surfaces isomorphic to $\mathbb{P}^1\times\mathbb{P}^1$, whose Hodge numbers are $h^{0,0}=1$, $h^{1,0}=0$, $h^{2,0}=0$, $h^{1,1}=2$. 
\end{itemize}

It follows that the Hodge diamond of the fixed locus is 
$$
\begin{array}{llll}
1+2k+k^2&&\\
&2g(1+k)&\\
2+2g^2+4k+2k^2&&g^2.
\end{array}
$$

Hence, using relations \eqref{eq: relations}, the Hodge numbers of $Z_S$, and so of any crepant resolution of $(S\times S)/(\iota_S\times\iota_S)$, are:

$$
\begin{array}{l}
h^{1,1}=20+2N-2N'+N^2,\\
h^{2,1}=2NN',\\
h^{3,1}=20-2N+2N'+N'^2,\\
h^{2,2}=204+4N^2-4NN'+4N'^2.
\end{array} 
$$
\subsection{The quotient of $Z_S$, birational to $Y_S$}\label{subsec: quotient of ZS}

The group $\langle \sigma, \iota_S\times \iota_S\rangle\simeq (\Z/2\Z)^2$ is contained in $\Aut(S\times S)$. In particular, the automorphism $\sigma$ commutes with $\iota_S\times \iota_S$ and thus it induces an automorphism of $(S\times S)/(\iota_S\times \iota_S)$, denoted by $\sigma$. Since the singular locus of $(S\times S)/(\iota_S\times \iota_S)$ is the image, under the quotient map, of $\Fix_{\iota_S}\times \Fix_{\iota_S}$, and since $\Fix_{\iota_S}\times \Fix_{\iota_S}$ is preserved by $\sigma$, the automorphism $\sigma$ extends to an automorphism $\sigma_Z$ of $Z_S$. Moreover, $\sigma_Z$ preserves the period on $Z_S$, because $\sigma$ preserves the $(4,0)$ form on $S\times S$.

So we have now a Calabi--Yau 4-fold $Z_S$ with an involution, $\sigma_Z$, which preserves the period of $Z_S$. 
\begin{proposition}\label{prop: ZS/sigma birat to YS}
The quotient 4-fold $Z_S/\sigma$ is birational to $Y_S$.
\end{proposition}
\proof The statement follows by the commutativity of the following diagram
$$
\xymatrix{ &&S\times S\ar[dr]\ar[dl]^{\pi}\ar[dd]\\
Z_S\ar[dr]\ar@{-->}[r]^{\beta_Z}&(S\times S)/(\iota_S\times\iota_S)\ar[dr]&&(S\times S)/\sigma\ar[dl]&S^{[2]}\ar@{-->}[l]\ar[dl]\\
&Z_S/\sigma_Z\ar@{-->}[r]&(S\times S)/\langle \sigma,\iota_S\times \iota_S\rangle&S^{[2]}/\iota_S^{[2]}\ar@{-->}[l]& Y_S\ar@{-->}[l]}
$$
where all dash arrows are the birational maps induced by the chosen crepant resolutions and the other arrows are quotient maps. 
 
\endproof

We will see that the 4-fold $Z_S/\sigma_Z$ is singular along surfaces, so it admits a crepant resolution $\widetilde{Z_S/\sigma_Z}$, which is a Calabi--Yau variety and is birational to $Y_S$. In particular, the Hodge numbers of $\widetilde{Z_S/\sigma_Z}$ coincide with those of $Y_S$.

Moreover, by \cite[Pag. 420]{Kawamata}, we know that the birational map $Y_S\dashrightarrow\widetilde{Z_S/\sigma_Z}$ can be decomposed into a sequence of flops.

We explicitly determine the fixed locus of the automorphism $\sigma'$ induced by $\sigma$ on $(S\times S)/(\iota_S\times \iota_S)$. Let $\pi:S\times S\ra (S\times S)/(\iota_S\times \iota_S)$ be the quotient map. The automorphism $\sigma:S\times S\ra S\times S$ acts by sending $(p,q)$ to $(q,p)$. The points on $(S\times S)/(\iota_S\times \iota_S)$ are denoted by $(\overline{p},\overline{q})$, where $(\overline{p},\overline{q})$ is the common image of $(p,q)\in S\times S$ and $(\iota_S(p),\iota_S(q))\in S\times S$ under the quotient map (i.e.  $(\overline{p},\overline{q})=\pi(p,q)=\pi(\iota_S(p),\iota_S(q))$).
Thus the condition $\sigma'(\overline{p},\overline{q})=(\overline{p},\overline{q})$ implies that either $p=q$ or $p=\iota(q)$. Hence the surfaces $\Sigma_1:=\pi(\{(p,p)|p\in S\})$ and $\Sigma_2:=\pi(\{(p,\iota_S(p))|p\in S\})$ are fixed by $\sigma'$.

Let $\beta_Z:Z_S\ra (S\times S)/(\iota_S\times \iota_S)$ be the blow up of $(S\times S)/(\iota_S\times \iota_S)$ in its singular locus. The two surfaces $\pi(\{(p,p)|p\in S\})$ and $\pi(\{(p,\iota_S(p))|p\in S\})$ intersect exactly in the curve $\pi(\Delta_{\Fix_{\iota_S}\times \Fix_{\iota_S}})$ inside the singular locus of $(S\times S)/(\iota_S\times \iota_S)$. A local analysis shows that $\sigma_Z$ fixes two disjoint surfaces in $Z_S$ (which are the strict transforms of the surfaces $\Sigma_1$ and $\Sigma_2$).

\subsection{The other $2:1$ quotient of $S\times S$} \label{subsec: other 2 quotient of SxS}
The group  $\langle \sigma, \iota_S\times \iota_S\rangle\simeq (\Z/2\Z)^2\subset\Aut(S\times S)$ contains three distinct copies of $\Z/2\Z$: the one generated by $\sigma$, which gives rise to $S^{[2]}$, the one generated by $\iota_S\times \iota_S$, which gives rise to $Z_S$, and the one generated by $\sigma\circ (\iota_S\times \iota_S)$.
\begin{proposition}\label{prop: other hk quotient of SxS}
The blow up of $(S\times S)/\left(\sigma\circ (\iota_S\times \iota_S)\right)$ in its singular locus is isomorphic to $S^{[2]}$.
\end{proposition}
\proof

We consider the isomorphism $\phi=\iota_S\times \id_S:S\times S\rightarrow S\times S$; given the graph $\Gamma_{\iota_S}$ of $\iota_S$, we have $\phi^{-1}(\Gamma_{\iota_S})=\Delta_S$, the diagonal in $S^2$. Since blow-up commutes with flat base-change, we obtain the following commutative diagram:
\[
\xymatrix{
\mathrm{Bl}_{\Delta_S}S^2\ar@{->}[r]^{\tilde{\phi}}\ar@{->}[d]&\mathrm{Bl}_{\Gamma_{\iota_S}}S^2\ar@{->}[d]\\
S^2\ar@{->}[r]^{\phi}& S^2.
}
\]
We have thus deduced the existence of an isomorphism between the two blow-ups. 
On the other hand, we remark that $\mathrm{Fix}(\sigma)=\Delta_S$ and $\mathrm{Fix}(\sigma\circ \iota_S\times \iota_S)=\Gamma_{\iota_S}$, hence we get induced involutions on the blow-ups, that we still denote as on $S^2$. Moreover, $\phi\circ\sigma=(\sigma\circ \iota_S\times \iota_S)\circ\phi$, so everything is equivariant and induces an isomorphism also between the smooth quotients, i.e. we obtain
\[
S^{[2]}\xrightarrow{\tilde{\phi}}\mathrm{Bl}_{\Gamma_{\iota_S}}S^2/(\sigma\circ \iota_S\times \iota_S).
\]
\endproof

\subsection{The Picard group of $Y_S$}\label{subsec: the picard group of YS}
Since $Y_S$ is a Calabi--Yau variety, $H^2(Y_S,\Z)=\Pic(Y_S)$, so in order to determine a $\Q$-basis of $Y_S$ it suffices to find a $\Q$-basis of $H^{2}(Y_S,\Z)$. This follows directly from the previous description of $H^{1,1}(Y_S)$.
Let us assume that $S$ is generic among the K3 surfaces admitting a non symplectic involution $\iota_S$ with a given fixed locus: this is equivalent to require that $\rk(\NS(S))=\dim H^2(S)^{\iota_S}$, i.e. $\rho(S)=\rk \NS(S)=r$. Let us denote by $D_S^{(1)},\ldots D_S^{(r)}$ a basis of $\NS(S)=\Pic(S)$. Let us now consider $S\times S$. A basis of $H^{1,1}(S)^{\iota_S}\otimes H^{0,0}(S)$ is given by $D_S^{(i)}\times [S]$ for $i=1,\dots,r$. Each divisor $D_S^{(i)}\times [S]$ is sent by $\sigma^*$ to $[S]\times D_S^{(i)}\in H^{0,0}(S)\otimes H^{1,1}(S)^{\iota_S}$. In particular, the class $D_S^{(i)}\times [S]+[S]\times D_S^{(i)}$ is preserved by $\langle \iota_S\times \iota_S, \sigma\rangle$ and thus it corresponds to a class in $\Pic(Y_S)$, denoted by $D_Y^{(i)}$. The other generators of $\Pic(Y_S)$ come from the desingularization of the quotients of order two.

We recall once again the construction of our desingularizations by a diagram:
$$
\xymatrixcolsep{3pc}\xymatrix {&\widetilde{S\times S}\ar[r]^{\beta_{\Fix(\iota_S\times \iota_S)}}\ar[d]_{/\iota_S\times \iota_S}&S\times S\ar@{-->}[dl]\ar@{-->}[dr]\ar[dd]& \widetilde{\widetilde{S\times S}}\ar[l]_{\beta_\Delta}\ar[d]_{/\sigma=:\pi_1}\\
\widetilde{Z_S}\ar[r]_{\beta_{\Fix(\sigma_Z)}}\ar[d]_{/\sigma _Z }&Z_S&&S^{[2]}&\widetilde{S^{[2]}}\ar[l]_{\beta_{\Fix(\iota_S^{[2]})}}\ar[d]_{/\iota_S^{[2]}=:\pi_2}\\
\widetilde{Z_S/\sigma_Z}\ar[rr]&&(S\times S)/\langle\iota_S\times \iota_S,\sigma\rangle&&Y_S\ar[ll]
}
$$
We will denote the quotient maps as follows: $\pi_1:\widetilde{S\times S}\ra S^{[2]}$ and $\pi_2: \widetilde{S^{[2]}}\ra Y_S$. We study the divisors introduced by the blow ups $\beta_{\Delta}$ and $\beta_{\Fix(\iota_S^{[2]})}$ and identify the ones which are preserved by the quotient maps.  

The blow up $\beta_{\Delta}$ introduces one divisor which is the exceptional divisor over the diagonal and is also the branch divisor of the quotient $\pi_1:\widetilde{S\times S}\ra S^{[2]}$. The natural involution $\iota_S^{[2]}$ preserves this divisor, called the exceptional divisor of $S^{[2]}$, (cfr. \cite[Theorem 1]{BS}), so its image under the quotient map $\pi_2$ is a divisor in $\Pic(Y_S)$, which will be denoted by $E_{\Delta}$. The other divisors of $Y_S$ come from the blow up $\beta_{\Fix(\iota_S^{[2]})}$. They are:\begin{itemize}\item $E_{C\times C}$, the exceptional divisor over $C^{[2]}$ (which is a surface fixed by $\iota_S^{[2]}$ and is the image of $C\times C\subset S\times S$ under the quotient map $\pi_1$);
\item 
$E_{R_i\times R_i}$ for $i=1,\ldots, k$, the exceptional divisor over $R_i^{[2]}$ (which is a surface fixed by $\iota_S^{[2]}$ and is the image of $R_i\times R_i\subset S\times S$ under the quotient map $\pi_1$); 
\item   $E_{C\times R_i}$, for $i=1,\ldots, k$: $C\times R_i\subset S\times S$ is a surface which is sent to $R_i\times C$ by $\sigma$, the common image of these surfaces is fixed by $\iota_S^{[2]}$ and so it is blown up by $\beta_{\Fix(\iota_S^{[2]})}$, and $E_{C\times R_i}$ is the exceptional divisor of this blow up;
\item $E_{R_i\times R_j}$, for $i,j=1,\ldots k$, $i<j$: $R_i\times R_j\subset S\times S$ is a surface which is sent to $R_j\times R_i$ by $\sigma$, and the common image of these surfaces is fixed by $\iota_S^{[2]}$ and so it is blown up by $\beta_{\Fix(\iota_S^{[2]})}$; $E_{R_i\times R_j}$, with $i<j$, is the exceptional divisor of this blow up;
\item $E_{S/\iota_S}$: $\iota_S^{[2]}$ fixes a surface inside the exceptional divisor of $S^{[2]}$, which is given by the image of points $(p,\iota_S(p))\in S\times S$; this surface is isomorphic to $S/\iota_S$ and so $\beta_{\Fix(\iota_S^{[2]})}$ introduces an exceptional divisor on it, denoted by $E_{S/\iota_S}$.\end{itemize}
\begin{prop}
With the same notation as above, let $\mathcal{S}=\{ D_Y^{(h)}$,  $E_{\Delta}$, $E_{S/\iota_S}$, $E_{C\times C}$, $E_{C\times R_i}$, $E_{R_i\times R_i}$, $E_{R_i\times R_j}\}$, where $h=1,\ldots ,r$, $i,j=1,\ldots k$ and $i<j$. Then $\mathcal{S}$ is a $\Q$-basis of $\NS(Y_S)$,  $\NS(Y_S)\otimes \Q$ can be identified with $\NS(\widetilde{Z_S/\sigma_Z})\otimes \Q$ and so $\mathcal{S}$ is a  $\Q$-basis of $\NS((\widetilde{Z_S/\sigma_Z})$.
The divisors 
$$B_{\iota_S^{[2]}}:=\sum_{i=1}^k\left(E_{C\times R_i}+E_{R_i\times R_i}+\sum_{j=i+1}^k E_{R_i\times R_j}\right)+E_{C\times C}+E_{S/\iota_S}$$
$$B_{\sigma_Z}:=E_{\Delta}+E_{S/\iota_S}$$
and $B_{\iota_S^{[2]}}+B_{\sigma_Z}$ are 2-divisible in $\NS(Y_S)$ and indeed they are associated to three different double covers of $Y_S$. \end{prop}
\proof
By construction $\mathcal{S}\subset \NS(Y_S)$ and the divisors in $\mathcal{S}$ are not linearly equivalent. The cardinality of $\mathcal{S}$ coincides with $h^{1,1}(Y_S)$, so $\mathcal{S}$ is a $\Q$-basis of $\NS(Y_S)$. 

We observe that since $\widetilde{Z_S/\sigma_Z}$ and $Y_S$ are related by a sequence of flops, they are isomorphic in codimension 1 and this allows one to identify the same basis $\mathcal{S}$ as a $\Q$-basis of $\NS(\widetilde{Z_S/\sigma_S})$. Indeed, the divisors $D_Y^{(h)}$, $h=1,\ldots ,r$ are induced by the cohomology of $S\times S$ and do not depend on the desigularization that we are considering. The other divisors come from the blow ups $\beta_{\Fix(\iota_S\times \iota_S)}$ and $\beta_{\Fix(\sigma)}$. The first blow up introduces exceptional divisors over the curves $C\times C$, $R_i\times R_i$, $R_i\times R_j$, $C\times R_i$, $R_i\times C$. The exceptional divisors over $C\times C$ and $R_i\times R_i$ are preserved by $\sigma$. The exceptional divisors over $R_i\times R_j$ (resp. $C\times R_i$) are identified with the ones over $R_j\times R_i$ (resp. $R_i\times C$) by the quotient by $\sigma_Z$. This 
gives the divisors $E_{C\times C}$, $E_{C\times R_i}$ for $i=1,\ldots, k$, $E_{R_i\times R_i}$ for $i=1,\ldots, k$, $E_{R_i\times R_j}$ for $i,j=1,\ldots k$, $i<j$ on $Y_S$. The blow up $\beta_{\Fix(\sigma_Z)}$ introduces two other divisors over the fixed locus of $\sigma_Z$ and we already proved that $\Fix_{\sigma_Z}$ consists of two surfaces: the strict transforms of the images $\Sigma_1$ and $\Sigma_2$ of $\lbrace(p,p)\in S\times S\rbrace$ and of $\lbrace(p,\iota_S(p))\in S\times S\rbrace$. We conclude that $\beta_{\sigma_Z}$ introduces the divisors $E_\Delta$ and $E_{S/\iota_S}$ on $Y_S$. Moreover, this shows that the support of the divisor $B_{\sigma_Z}$ is the branch locus of the $2:1$ cover $\widetilde{Z_S}\ra \widetilde{Z_S/\sigma_Z}\sim Y_S$. Hence it is a 2-divisible divisor on $Y_S$.

The support of the divisor $B_{\iota_S^{[2]}}$ is exactly the branch locus of the $2:1$ cover $\pi_2:\widetilde{S^{[2]}}\ra Y_S$.

Since the divisors $B_{\iota_S^{[2]}}$ and $B_{\sigma_Z}$ are 2-divisible, also $B_{\iota_S^{[2]}}+B_{\sigma_Z}$ is 2-divisible
in $\Pic(Y_S)$. This means that $Y_S$ admits a 2:1
cover branched along $\coprod_{i=1}^{k}\left(E_{C\times R_i}\right)$ $\coprod_{i=1}^{k}\left(E_{R_i\times R_i}\right)$ $\coprod_{i,j=1, i<j}^{k}\left(E_{R_i\times R_j}\right)$ $\coprod E_{C\times C}$ $\coprod E_{\Delta}$. This cover is naturally birational to $(S\times S)/(\iota_S\times \iota_S\circ\sigma)$.

Since there are divisors which are 2-divisible in $\NS(Y_S)$, but which are not 2- divisible in $\mathcal{S}$, we conclude that $\mathcal{S}$ can not be a $\Z$-basis of $\NS(Y_S)$.\endproof

\subsection{Remarks on complex deformations and mirror symmetry}\label{subsec: remarks on mirror symmetry}
\subsubsection{Complex deformations}

We make some remarks on the dimensions of the families of 4-folds which we are constructing.

\begin{proposition}
Table \ref{table:cpx defs} below lists the dimensions of local complex deformations of the families of $4$-folds constructed in the previous sections.
\bgroup
\def\arraystretch{1.3}
\begin{table}[h]\label{table:cpx defs}
 \begin{tabular}{|c|c|c|}
 \hline 
 Object & Dimension of complex deformations space & Dimension in terms of $(N,N')$\\
 \hline 
 $(S,\iota_S)$ & $\dim \left((H^{1,1}(S)^{\iota_S})^{\perp}\cap H^{1,1}(S)\right)$ & $10-N+N'$\\
 \hline 
 $(S\times S,\iota_S\times\iota_S)$ & $\dim \left((H^{1,1}(S)^{\iota_S})^{\perp}\cap H^{1,1}(S)\right)^{\oplus 2}$ & $20-2N+2N'$\\
 \hline
$Z_S$ & $h^{3,1}(Z_S)$ & $20-2N+2N'+N'^2$\\
\hline
$(S^{[2]},\iota^{[2]})$ & $\dim \left((H^{1,1}(S^{[2]})^{\iota_S^{[2]}})^{\perp}\cap H^{1,1}(S^{[2]})\right)$ & $10-N+N'$\\
\hline
$(S\times S,\iota_S\times\iota_S,\sigma)$ & $\dim \left((H^{1,1}(S)^{\iota_S})^{\perp}\cap H^{1,1}(S)\right)$ & $10-N+N'$\\
\hline
$Y_S$ & $h^{3,1}(Y_S)$ &  $(20-2N+N'+N'^2)/2$\\
\hline 
 \end{tabular} 
 \vspace{.01cm}
\end{table}
\egroup
\end{proposition}
\proof
The fact that the dimension of the deformation space of a pair $(X,f)$, for $X$ hyperk\"ahler and $f\in\Aut(X)$ a non symplectic involution, equals the dimension of $H^{1,1}(X)\cap (H^{1,1}(X)^{f})^{\perp}$ is proven in \cite[\S 4]{bcs14} and implies the statement both for $(S,\iota_S)$ and for $(S^{[2]}, \iota_S^{[2]})$. 

In order to describe $\Def(S\times S,\iota_S\times \iota_S)$ and $\Def(S\times S,\iota_S\times \iota_S,\sigma)$ one observes that the same proof of loc.\ cit.\ yields, for any smooth complex manifold $X$ such that $H^0(X,T_X)=0$ and $c_1(X)=0$ and any automorphism $f\in\Aut(X)$, that the dimension of the family $\Def(X,f)$ coincides with $\dim H^1(X,T_X)^{df}$. Since $T_{S\times S}\cong T_S\boxplus T_S$, we have indeed that $H^0(S\times S,T_{S\times S})=0$, so that by \cite[Theorem 14.10]{grosscalabi} 
$\Def(S\times S)\cong H^1(S\times S,T_{S\times S})$ as germs over $0$; moreover, $H^1(S\times S,T_{S\times S})^{d(\iota_S\times\iota_S)}\cong (H^1(S,T_S)^{d\iota_S})^{\oplus 2}$, hence the statement for $\Def(S\times S,\iota_S\times \iota_S)$. Iteration of this reasoning finally gives the deformations of $(S\times S,\iota_S\times\iota_S,\sigma)$.

Finally, the fact that local complex deformations of a smooth Calabi--Yau manifold $X$ are given by an open set inside $H^{n-1,1}(X)$ is the well-known Tian--Todorov's theorem for smooth K\"ahler manifolds with trivial canonical bundle \cite[Theorem 6.8.1]{grosscalabi}.

The third column is now an easy consequence of the previous computations.
\endproof

We will say that a Calabi--Yau 4-fold is of Borcea--Voisin type if it is the desingularization of the quotient $(S_1\times S_2)/(\iota_1\times \iota_2)$ where $S_i$ are K3 surfaces and $\iota_i$ is a non symplectic involution on $S_i$ and we generalize the definition \cite[Definition 3.6]{CG} saying that a Borcea--Voisin maximal family is a family of Calabi--Yau 4-folds
such that the generic member of this family is of Borcea--Voisin type.

\begin{corollary} Given a pair $(S,\iota_S)$: \begin{itemize}\item $\dim(\Def(Z_S))\geq\dim(\Def(S\times S,\iota_S\times \iota_S))$,  and the equality holds if and only if $N'=0$ (i.e.\ if  $\iota_S$ fixes on $S$ only rational curves);
\item
$\dim(\Def(S^{[2]},\iota_S^{[2]}))=\dim(\Def(S\times S,\iota_S\times \iota_S,\sigma))$, but $\dim(\Def(S^{[2]}))=\dim(\Def(S,\iota_S))+1$;
\item
$\dim(\Def(Y_S))\geq\dim(\Def(S^{[2]},\iota_S^{[2]}))=\dim(\Def(S\times S,\iota_S\times \iota_S,\sigma))$. The equality holds if and only if either $N'=0$ (i.e.\ the fixed locus $\Fix_{\iota_S}(S)$ is rigid) or $N'=1$.
\item $\dim(\Def(Y_S))\leq\dim(\Def(Z_S))$. The equality holds if and only if $\iota_S$ fixes exactly $10$ rational curves, i.e. if $(S,\iota_S)$ is rigid. 
\end{itemize}
\end{corollary}

In particular, it follows from $\dim(\Def(Z_S))\geq\dim(\Def(S\times S,\iota_S\times \iota_S)$ that not all the Calabi--Yau 4-folds which deform $Z_S$ are of Borcea--Voisin type. Indeed the deformations of $Z_S$ are all of Borcea--Voisin type if and only if the fixed locus of $\iota_S$ on $S$ is rigid. In this case the family of $Z_S$ is a Borcea--Voisin maximal family, in analogy with \cite[Proposition 3.7]{CG}. This implies that the variation of its period depends only on the variation of the period of $S\times S$ in the family $\Def(S\times S,\iota_S\times \iota_S)$.

By $\dim(\Def(S^{[2]},\iota_S^{[2]}))=\dim(\Def(S\times S,\iota_S\times \iota_S,\sigma))$ it follows that all the deformations of  $S^{[2]}$ which preserve the non symplectic involution $\iota_S^{[2]}$ are dominated by $S\times S$, but there is one more deformation of $S^{[2]}$ if we do not require that $\iota_S^{[2]}$ deforms with $S^{[2]}$.

By $\dim(\Def(Y_S))\geq\dim(\Def(S^{[2]},\iota_S^{[2]}))$, it follows that not all the deformations of $Y_S$ are dominated by $S^{[2]}$ and by $S\times S$. However, again, if the fixed locus of $\iota_S$ is rigid on $S$, then all the deformations of $Y_S$ are obtained both as the crepant resolution of $S^{[2]}/\iota_S^{[2]}$ and as crepant resolution of $(S\times S)/\langle\iota_S\times \iota_S,\sigma\rangle$. In this case the variation of the period of $Y_S$ depends only on the variation of the period of $S$ in $\Def(S,\iota_S)$. This is the analogue of the proposition \cite[Proposition 3.7]{CG}. A little bit more surprising is the fact that also if $N'=1$, i.e. if the fixed locus of $\iota_S$ contains a curve of genus 1 (which a priori can be deformed), $\dim(\Def(Y_S))=\dim(\Def(S^{[2]},\iota_S^{[2]}))=\dim(\Def(S\times S,\iota_S\times \iota_S,\sigma)).$  

Since $\dim(\Def(Y_S))\leq\dim(\Def(Z_S)$, all the deformations of $Y_S$ are dominated by deformations of $Z_S$. We observe that a generic deformation of $Z_S$ does not necessarily admit the automorphism $\sigma_Z$ needed to construct $Y_S$ as quotient.

\subsubsection{Mirror symmetry}

Here we discuss the mirror symmetry of $Y_S$ and $Z_S$, at least at the level of the Hodge diamond. For this reason we recall here the dimension of the space of small deformations of the K\"ahler structure, since we want to compare it with the dimension of the complex deformations, given in Table \ref{table:cpx defs}:

\bgroup
\def\arraystretch{1.3}
\begin{table}[h]\label{table:khaeler defs}
 \begin{tabular}{|c|c|c|}
 \hline 
 Object & Dimension of K\"ahler deformations space&Dimension in terms of $(N,N')$ \\
 \hline
$Z_S$ & $h^{1,1}(Z_S)$& $20+2N-2N'+N^2$\\
\hline
$(S^{[2]},\iota^{[2]})$ & $\dim \left((H^{1,1}(S^{[2]})^{\iota_S^{[2]}})\right)$ & $11+N-N'$\\
\hline
$Y_S$ & $h^{1,1}(Y_S)$ &  $\left(24+3N-2N'+N^2\right)/2$\\
\hline 
 \end{tabular} 
 \vspace{.01cm}
\end{table}
\egroup
In the following we will prove the following:
\begin{corollary} Let $\check{S}$ (resp. $M_{S^{[2]}}$) be the K3 surface (resp. \hk of K3 type) that is the lattice theoretic mirror of the K3 surface $S$, (resp. of the \hk $S^{[2]}$). Let $\check{Y_S}$ be a Calabi--Yau 4-fold whose Hodge diamond is mirror to the one of $Y_S$. Then $\check{Y_S}\neq Y_{\check{S}}$ and $\check{Y_S}\neq \widetilde{M_{S^{[2]}}/\iota_M}$, where $\iota_M$ is the mirror involution of $\iota_S^{[2]}$ and $\widetilde{M_{S^{[2]}}/\iota_M}$ is the Calabi--Yau constructed as in Section \ref{subsec: crepant resoltuion}. \end{corollary}

Under a mild condition on the N\'eron--Severi group, a lattice theoretic mirror symmetry between K3 surfaces is defined by Dolgachev in \cite{dolgachevmirror}, extending work by Pinkham and Nikulin. In particular, if $S$ is a generic K3 surface with a non symplectic involution $\iota_S$ with a given fixed locus which contains exactly one curve of positive genus, then $S$ admits a mirror symmetric K3 surface $\check{S}$ with a non symplectic involution $\check{\iota_S}$, as showed by Voisin in \cite{Voisin}. Let $N$ and $N'$ be the invariants of the fixed locus of  $\iota_S$ on $S$ and let $\check{N}$ and $\check{N'}$ be the invariants of the fixed locus of  $\check{\iota_S}$ on $\check{S}$, we have $N=\check{N'}$, $N'=\check{N}$. It is immediate to check that this induces the mirror symmetry between the Hodge diamond of $Z_S$ and that of $Z_{\check{S}}$ (which is a crepant resolution of $(\check{S}\times 
\check{S})/(\check{\iota_S}\times \check{\iota_S})$), i.e. $h^{1,1}(Z_S)=h^{3,1}(Z_{\check{S}})$ and $h^{2,2}(Z_S)=h^{2,2}(Z_{\check{S}})$). This was already observed in \cite{Dillies}.

On the other hand, it is known that the mirror symmetry between $(S,\iota_S)$ and $(\check{S},\check{\iota_S})$ does not induce a lattice theoretic mirror symmetry between $S^{[2]}$ and $\check{S}^{[2]}$. This is because mirror symmetry for hyperk\"ahler manifolds works differently from the case of Calabi--Yaus (see Huybrechts' lecture notes for extensive explanations of this phenomenon \cite{HuybrechtsMirror}); indeed, the generalization of the lattice theoretic mirror symmetry between two hyperk\"ahler 4-folds of $\hsk$ type (see \cite{CamereMirror} for further details) induces mirror symmetry between families of 4-folds of $\hsk$ type endowed respectively with natural and non-natural involutions whose invariants sublattices' ranks satisfy the relation $\check{r}\leftrightarrow 21-\check{r}$. 

In view of these two results it is quite natural to ask whether the lattice theoretic mirror symmetry, either on K3 surfaces or on hyperk\"ahler varieties of $\hsk$ type, induces a mirror symmetry between the Calabi--Yau 4-folds $Y_S$. The answer is no in both cases: let us consider a pair $(S,\iota_S)$ and its mirror pair $(\check{S},\check{\iota_S})$. As we noticed this implies that $N=\check{N'}$, $N'=\check{N}$ and one can directly check on the Hodge numbers given in Theorem \ref{prop: HOdge diamond natural} that this transformation does not give the required relations, i.e. one obtains that in general $h^{1,1}(Y_S)\neq h^{3,1}(Y_{\check{S}})$ and $h^{2,2}(Y_S)\neq h^{2,2}(Y_{\check{S}})$. 

The situation is more complicated in the case of the mirror symmetry on hyperk\"ahler varieties of $\hsk$ type, since in order to compute the Hodge numbers of the 4-folds obtained as quotient of a deformation of $S^{[2]}$ by a non natural involution $i$, one needs a good description of the fixed locus (in order to compute the Hodge diamond of a resolution of $S^{[2]}/i$). In general this is not available, but it suffices to test the mirror symmetry in the unique explicit case that we have, in order to show that it does not hold.

The first possible test fails:  let $S$ be a K3 surface with a non symplectic involution $\iota_S$ such that the numbers associated to its fixed locus are $N=10$ and $N'=2$. Under these conditions $h^{3,1}(Y_S)=3$. The lattice theoretic mirror of the pair $(S^{[2]},\iota_S^{[2]})$ is the pair $(V,\iota_V)$ where $V$
 is a \hk 4-fold of K3 type with $NS(V)\simeq U$ and $\iota_V$ is the non natural symplectic involution on it described by Ohashi and Wandel in \cite{ow}. For this involution $t_{1,1}=2$ (indeed the subspace of $H^{1,1}(V)^{\iota_V}$ is the N\'eron--Severi group). Moreover Ohashi and Wandel proved that the fixed locus of this involution consists of 2 disjoint surfaces, so $h^{1,1}(\widetilde{V/\iota_V})=2+2=4\neq 3=h^{3,1}(Y_S)$.

In a certain sense it is not very surprising that the mirror symmetry of $Y_S$ can not be deduced by the mirror symmetry of $S$ or of $S^{[2]}$. Indeed, the space of the complex (resp. K\"ahler) deformations of $Y_S$ includes not only deformations which come from $S$ but also some deformations coming from the deformations of $Z_S$ and some coming from deformations of $S^{[2]}$. Since the two mirror constructions, of $Z_S$ and of $S^{[2]}$ respectively, are not compatible, it seems reasonable that none of these two gives the right one for $Y_S$.

\section{Quotients of $S\times S$ and covers of these quotients}\label{sec: quotients}

Since the 4-folds $S^{[2]}$, $Y_S$, $Z_S$ and $\widetilde{Z_S/\sigma_S}$ are obtained as desingularizations of the quotients of $S\times S$ by an automorphism of $S\times S$, we now consider a subgroup of $\Aut(S\times S)$ and the quotients of $S\times S$ by this subgroup. From our geometric setting a natural choice for the subgroup involves $\iota_S\times \iota_S$, $\sigma$ and $\iota_S\times \id$. By Proposition \ref{prop: Oguiso amximal autom group}, due to Oguiso,  this is also the "maximal" choice for a generic K3 surface with an involution. Indeed, if $S$ is general among the K3 surfaces admitting an involution (i.e. it is generic either among the $\langle 2\rangle$-polarized K3 surfaces or among the $U$-polarized K3 surfaces or among the $U(2)$-polarized K3 surfaces), then $\Aut(S)=\langle \iota_S\rangle$ and thus the subgroup described coincides with $\Aut(S\times S)$. 

For every K3 surface admitting a non symplectic involution $\iota_S$, we will consider the quotients by $\mathcal{D}_8=\langle\iota_S\times \id,\sigma\rangle$ and by its subgroups. Let $W$ be the smooth surface $S/\iota_S$. We observe that $\left(S\times S\right)/\langle\iota_S\times \id,\sigma\rangle\simeq W^{(2)}$. Then we describe the singular models of $S^{[2]}$, $Y_S$, $Z_S$  as covers of $S\times S/\langle\iota_S\times \id,\sigma\rangle$ under the assumption that the fixed locus of $\iota_S$ is connected (see Proposition \ref{prop: description of 4-folds as double covers}). In a very particular case, i.e. if $\iota_S$ is a fixed point free involution, and so $W$ is an Enriques surface, this allows us to prove that $\widetilde{Z_S/\sigma_Z}$ is in fact isomorphic (and not only birational) to $Y_S$ (see Proposition \ref{prop:quoz_enriq}) and  to show that $Y_S$ is indeed the universal cover of $W^{[2]}$ mentioned in \cite[Theorem 3.1]{Oguiso-Schroer}.

\begin{prop}{\cite[Theorem 4.1]{Oguiso-Malte-question}}\label{prop: Oguiso amximal autom group}
The automorphism group of $S\times S$ is $\Aut(S\times S)\simeq \langle \sigma\rangle\rtimes Aut(S)^2$.\end{prop}

\begin{rem}{\rm If $\Aut(S)\simeq \Z/2\Z$, then $\Aut(S\times S)\simeq \mathcal{D}_8$.}\end{rem}

The group $\mathcal{D}_8\simeq \langle \iota_S\times \id,\id\times \iota_S, \sigma\rangle$ contains the following elements: $$\begin{array}{ll}g_1:=\id\times \id:(p,q)\mapsto (p,q)&g_2:=\sigma:(p,q)\mapsto(q,p)\\
g_3:=\iota_S\times \id:(p,q)\mapsto (\iota_S(p),q)&g_4:=\sigma\circ(\iota_S\times \id):(p,q)\mapsto (q,\iota_S(p))\\ 
g_5:=\id\times\iota_S:(p,q)\mapsto (p,\iota_S(q))&g_6:=\sigma\circ(\id\times\iota_S):(p,q)\mapsto(\iota_S(q),p)\\ 
g_7:=\iota_S\times\iota_S:(p,q)\mapsto(\iota_S(p), \iota_S(q))& g_8:=\sigma\circ(\iota_S\times\iota_S):(p,q)\mapsto(\iota_S(q),\iota_S(p))
\end{array}
$$
We observe that $g_2$, $g_3$, $g_5$, $g_7$, $g_8$ have order two, while $g_4$ and $g_6$ have order 4  ($g_4^3=g_6$) and their square is $g_7$.
The subgroup $\langle g_7\rangle$ is the center of the group, in particular it is normal. The other normal subgroups are $\langle g_7, g_2\rangle\simeq (\Z/2\Z)^2$, $\langle g_3,g_7\rangle\simeq (\Z/2\Z)^2$ and $\langle g_4\rangle\simeq \Z/4\Z$. 

We denote by $\overline{g_i}$ the automorphisms induced by $g_i$ on the quotients of $S\times S$ by a certain subgroup of $\mathcal{D}_8$; we obtain the following diagram, where all the arrows are quotients of order two:
\begin{align}\label{eq: diagram quotient by D8}\xymatrix{ &&S\times S\ar[dl]^{/g_8}\ar[dll]_{\gamma_1:=/g_2}\ar[d]_{\alpha_1:=/g_7}\ar[dr]_{/g_5}\ar[drr]^{/g_3}\\ 
S^{(2)}\ar[dr]_{\gamma_2:=/\overline{g_7}}&(S\times S)/(\sigma\circ(\iota_S\times \iota_S))\ar[d]_{/\overline{g_2}}&(S\times S)/(\iota_S\times \iota_S)\ar[dl]^{\beta_2:=/\overline{g_2}}\ar[dr]_{\alpha_2:=/\overline{g_3}}&S\times W\ar[d]_{/\overline{g_7}}&W\times S\ar[dl]_{/\overline{g_7}}\\
&S^{(2)}/\iota_S^{(2)}\ar[dr]_{\beta_3:=/\overline{g_3}}&&W\times W\ar[dl]_{\alpha_3:=/\overline{g_2}}&\\ &&W^{(2)}}\end{align}

The 4-folds $S^{(2)}$, $(S\times S)/(\sigma\circ(\iota_S\times \iota_S))$, $(S\times S)/(\iota_S\times \iota_S)$ and $W^{(2)}$ are singular along surfaces and admit the crepant resolutions $S^{[2]}$, $\widetilde{(S\times S)/(\sigma\circ(\iota_S\times \iota_S))}$, $Z_S$ and $W^{[2]}$ respectively. We already proved that $\widetilde{(S\times S)/(\sigma\circ(\iota_S\times \iota_S))}$ is isomorphic to $S^{[2]}$. The singular quotient $S^{(2)}/\iota_S^{(2)}$ is birational to the Calabi--Yau 4-fold $Y_S$. All the other 4-folds which appear in the diagram are smooth.

The $4:1$ map $(S\times S)/(\iota_S\times \iota_S)\ra W^{(2)}$ is the quotient map by the group $(\Z/2\Z)^2\simeq \mathcal{D}_8/g_7$.
There are also some $4:1$ maps which are induced by this diagram, but which are not quotient maps (i.e. the target space is not the quotient of the domain by the action of a group of order 4 defined over the domain): by the previous diagram both $S^{(2)}\ra W^{(2)}$ and $(S\times S)/(\sigma\circ\iota_S\times \iota_S)\ra W^{(2)}$ have order 4. 

There is a $2:1$ quotient map $S^{(2)}/\iota_S^2$ to $W^{(2)}$.

There is a map $(S\times S)/(\iota_S\times \iota_S)\ra W$, obtained by composing $(S\times S)/(\iota_S\times \iota_S)\ra W\times W$ with the projection on one factor $W\times W\ra W$. The generic fiber of $(S\times S)/(\iota_S\times \iota_S)\ra W$ is isomorphic to $S$.

\subsection{Covers}\label{subsec: covers}

We recall that $W\simeq S/\iota_S$ is a smooth surface.
Let us consider $W^{(2)}$, i.e. $(W\times W)/\sigma$. This is a 4-fold, singular along a surface isomorphic to $W$, image of the diagonal under the quotient map $W\times W\ra W^{(2)}$.

By diagram \eqref{eq: diagram quotient by D8}, it is clear that $W^{(2)}$ is the quotient of $S\times S$ by the group $\langle \iota\times \id, \sigma\rangle\simeq \mathcal{D}_8$.

Here we want to describe the 4-folds involved in our construction as covers, starting by the 4-fold $W^{(2)}$.

We will consider the following quotient maps:
\begin{align}\label{eq: notation}S\times S\stackrel{\alpha_1}{\ra} (S\times S)/(\iota\times \iota)\stackrel{\alpha_2}{\ra} W\times W\stackrel{\alpha_3}{\ra} W^{(2)};\end{align}
$$S\times S\stackrel{\alpha_1}{\ra} (S\times S)/(\iota\times \iota)\stackrel{\beta_2}{\ra} S^{(2)}/\iota^{(2)}\stackrel{\beta_3}{\ra} W^{(2)};$$
$$S\times S\stackrel{\gamma_1}{\ra} S^{(2)}\stackrel{\gamma_2}{\ra} S^{(2)}/\iota^{(2)}\stackrel{\beta_3}{\ra} W^{(2)};$$

and the following subspaces in $S\times S$:
$$\Delta_S:=\{(p,p)|p\in S\}\simeq S;$$ 
$$\Gamma_S:=\{(p,\iota_S(p))|p\in S\}\simeq S;$$ 
$$\Delta_C:=\{(p,p)|p\in \Fix_{\iota}\}\simeq \Fix_{\iota};$$ 
(the notation comes from the special case where $\Fix_{\iota}$ consists of a curve $C$)
$$B_C:=\{(p,q)|p\in \Fix_{\iota}(S),\  q\in \Fix_{\iota}(S)\}\simeq \Fix_{\iota}(S)\times \Fix_{\iota}(S);$$
$$T_1:=\{(p,q)|p\in \Fix_{\iota}(S),\ q\in S\}\simeq \Fix_{\iota}(S)\times S;$$ 
$$T_2:=\{(p,q)|p\in S,\ q\in \Fix_{\iota}(S)\}\simeq S\times \Fix_{\iota}(S).$$ 

The map $\pi_{tot}:S\times S\ra W^{(2)}$ coincides with each of the compositions: $\alpha_3\circ\alpha_2\circ\alpha_1$, $\beta_3\circ\beta_2\circ\alpha_1$ and
$\beta_3\circ\gamma_2\circ\gamma_1$.

We also remark that $\Sing(W^{(2)})=\pi_{tot}(\Delta_S)=\pi_{tot}(\Gamma_S)$.
In the following we will assume that $\Fix_{\iota}(S)$ consists of one smooth irreducible curve and we will prove the following proposition.

\begin{prop}\label{prop: description of 4-folds as double covers} Let $(W,C)$ be a pair such that: $W$ is a smooth surface; $C$ is a smooth curve in $W$; there exists the double cover of $W$ branched on $C$ and it is a K3 surface $S$. Let $\iota_S$ be the cover involution of $S\ra W$.

Let us denote by $T_{tot}$ the image of $(W\times C)\cup (C\times W)\subset W\times W$ under the quotient map $W\times W\ra W^{(2)}$.

The double cover of $W^{(2)}$ branched along $T_{tot}$ is a singular 4-fold $V_1$ which is singular along three surfaces $A_1$, $A_2$, $A_3$, all meeting transversally in the same curve, isomorphic to $C$. Two of the singular surfaces, say $A_1$ and $A_2$, are isomorphic to $W$, the third to $C^{(2)}$. There exists a crepant resolution of $V_1$ which is a Calabi--Yau 4-fold. The 4-fold $V_1$ is isomorphic to $S^{(2)}/\iota^{(2)}$ and is birational to $Y_S$.  

The double cover of $V_1$ branched in the union of $A_3\simeq C^{(2)}$ and one other singular surface ($A_i$, with either  $i=1$ or $i=2$) is a singular 4-fold $V_2$. The 4-fold $V_2$ is singular along a surface and the blow up of $V_2$ in its singular locus is a hyperk\"ahler 4-fold, isomorphic to $S^{[2]}$. In particular, $V_2\simeq S^{(2)}$.

The double cover of $V_1$ branched along the union of $A_1\simeq W$ and $A_2\simeq W$ is a singular 4-fold $V_3$. The 4-fold $V_3$ is singular along a surface and the blow up of $V_3$ in its singular locus is a Calabi--Yau 4-fold, isomorphic to $Z_S$. In particular $V_3\simeq (S\times S)/(\iota_S\times \iota_S)$.

The double cover of $V_2$ branched along its singular locus, coincides with the double cover of $V_3$ branched along its singular locus and it is $S\times S$.
\end{prop}
\proof The proof is based on the diagram \eqref{eq: diagram quotient by D8} and the details on the branch of the $2:1$ covers are given in Sections \ref{subsubsec: S2/i2}, \ref{subsubsec: S2} and \ref{subsubsec: SxS/ixi, I}, where we use the notation introduced in \eqref{eq: notation}.\endproof

\subsubsection{$V_1\simeq S^{(2)}/\iota_S^{(2)}$ as double cover of $W^{(2)}$}\label{subsubsec: S2/i2}

The threefold $T_{tot}:=\pi_{tot}(T_1)=\pi_{tot}(T_2)\subset W^{(2)}$ meets the singular locus of $W^{(2)}$ in the curve $\pi_{tot}(\Delta_C)$. Moreover we observe that it is a singular threefold and its singular locus is $\pi_{tot}(B_C)$. 

By diagram \eqref{eq: diagram quotient by D8}, the double cover of $W^{(2)}$ branched along $T_{tot}$ is the 4-fold $S^{(2)}/\iota_S^{(2)}$. The inverse image of $\Sing(W^{(2)})$ consists of two surfaces (which are $\beta_2(\alpha_1(\Delta_S))$ and $\beta_2(\alpha_1(\Gamma_S))$) meeting along the inverse image of $T_{tot}\cap \Sing(W^{(2)})$, which is the curve $\beta_3(\beta_2(\alpha_1(\Delta_C)))$). These two surfaces are singular for $S^{(2)}/\iota_S^{(2)}$ and intersects along a curve which is $\beta_2(\alpha_1(\Delta_C))$. They are $\beta_2(\alpha_1(\Delta_S))\simeq W$ and  $\beta_2(\alpha_1(\Gamma_S))\simeq W$. 

But the singularities of $S^{(2)}/\iota_S^{(2)}$ do not consist only of the surfaces $\beta_2(\alpha_1(\Delta_S))$ and $\beta_2(\alpha_1(\Gamma_S))$. Indeed, we already remarked that the branch threefold $T_{tot}$ is singular along a surface, which is $\pi_{tot}(B_C)\not\subset \Sing(W^{(2)})$. Hence, the $2:1$ cover $S^{(2)}/\iota_S^{(2)}\ra W^{(2)}$ is singular along the inverse image of this surface. Thus, we have a third singular surface in $S^{(2)}/\iota_S^{(2)}$, which is $\beta_2(\alpha_1(B_C))$ and it is isomorphic to $C^{(2)}$. This third singular surface intersects the other two singular surfaces in their common intersection, i.e. in the curve $\beta_2(\alpha_1(\Delta_C))$. The intersection among these three surfaces is transversal, since they are images of surfaces on $S\times S$ which generically have no common tangent directions.

\subsubsection{$V_2\simeq S^{(2)}$ as double cover of $V_1\simeq S^{(2)}/\iota_S^{(2)}$ (and thus as $4:1$ cover of $W^{(2)}$)}\label{subsubsec: S2}
The 4-fold $S^{(2)}/\iota_S^{(2)}$ is singular in three surfaces, which intersect in $\beta_2(\alpha_1(\Delta_C))$ and which are $\beta_2(\alpha_1(\Delta_S))$, $\beta_2(\alpha_1(\Gamma_S))$ and $\beta_2(\alpha_1(B_C))$. The same surfaces can be described also as $\gamma_2(\gamma_1(\Delta_S))$, $\gamma_2(\gamma_1(\Gamma_S))$ and $\gamma_2(\gamma_1(B_C))$.

Let us now consider the double cover of $S^{(2)}/\iota_S^{(2)}$ branched along $\gamma_2(\gamma_1(\Gamma_S))\cup\gamma_2(\gamma_1(B_C))$. By diagram \eqref{eq: diagram quotient by D8}, we obtain a 4-fold which is $S^{(2)}$. It is singular in $\gamma_1(\Delta_S)\simeq S$, indeed the quotient map $\gamma_2:S^{(2)}\ra S^{(2)}/\iota_S^{(2)}$ restricts to a $2:1$ map between $\gamma_1(\Delta_S)\simeq S$ and $\gamma_2(\gamma_1(\Delta_S))\simeq W$, branched along $\gamma_2(\gamma_1(\Delta_C))$.

\subsubsection{$V_3\simeq (S\times S)/(\iota_S\times \iota_S)$ as double cover of $V_1\simeq S^{(2)}/\iota_S^{(2)}$ (and thus as $4:1$ cover of $W^{(2)}$, via $S^{(2)}/\iota_S^{(2)}$)}\label{subsubsec: SxS/ixi, I}
Similarly to what we did in the previous paragraph, one can consider two surfaces among the three singular surfaces in $S^{(2)}/\iota_S^{(2)}$ and consider the double cover branched along these two surfaces. In particular, let us consider the double cover of $S^{(2)}/\iota_S^{(2)}$ branched along $\beta_2(\alpha_1(\Delta_S))\cup\beta_2(\alpha_1(\Gamma_S))$. We obtain a 4-fold which is $(S\times S)/(\iota_S\times\iota_S)$. It is singular in $\alpha_1(B_C)\simeq C\times C$, indeed the quotient map $\beta_2:(S\times S)/(\iota_S\times \iota_S)\ra S^{(2)}/\iota_S^{(2)}$ restricts to a $2:1$ map between $\alpha_1(B_C)\simeq S$ and $\beta_2(\alpha_1(B_C))\simeq C^{(2)}$, branched along $\beta_2(\alpha_1(\Delta_C))$.

\subsubsection{$V_2\simeq (S\times S)/\langle\iota_S\times \iota_S\circ \sigma\rangle)$ as double cover of $V_1\simeq S^{(2)}/\iota_S^{(2)}$ (and thus as $2^2:1$ cover of $W^{(2)}$)}
The third (and last) possible choice is to consider the double cover of $S^{(2)}/\iota_S^{(2)}$ branched along $\beta_2(\alpha_1(B_C))\cup\beta_2(\alpha_1(\Delta_S))$. We obtain a 4-fold which is $(S\times S)/\langle\iota_S\times\iota_S\circ \sigma\rangle$ and is isomorphic to $S^{(2)}\simeq V_2$ by Proposition \ref{prop: other hk quotient of SxS}.

\subsubsection{$(S\times S)/(\iota_S\times\iota_S)$ as $4:1$ cover of $W^{(2)}$, via $W\times W$}\label{subsubsec: SxS/ixi, II}

Let us consider the double cover of $W^{(2)}$ branched along $\Sing(W^{(2)})$. This is the smooth fourfold $W\times W$. Let us consider the (singular) threefold $T_{tot}:=\alpha_3(\alpha_2(\alpha_1(T_1)))=\alpha_3(\alpha_2(\alpha_1(T_2)))$, which intersects the singular locus of $W^{(2)}$ in $\alpha_3(\alpha_2(\alpha_1(\Delta_C)))\simeq C$. In the double cover $W\times W$, $T_{tot}$ splits in the two threefolds  $\alpha_2(\alpha_1(T_1))$ and $\alpha_2(\alpha_1(T_2))$, meeting along the surface $\alpha_2(\alpha_1(B_C))$ (which is in fact the surface mapped on the singularity of $T_{tot}$). The fourfold $(S\times S)/(\iota_S\times\iota_S)$ is the double cover of $W\times W$ branched along the union of the two threefolds over $\alpha_2(\alpha_1(T_1))$ and $\alpha_2(\alpha_1(T_2))$ (i.e. along the two inverse images of $T_{tot}$ in the double cover $W\times W\ra W^{(2)}$). These two threefolds meet along the surface ($\alpha_2(\alpha_1(B_C))\simeq C\times C$ , thus the double cover $(S\times S)/(\iota_S\times \
iota_S)$ is 
singular along the surface inverse image of $\alpha_2(\alpha_1(B_C))\simeq C\times C$.

\subsection{A very special case: $\iota_S$ is an Enriques involution}\label{subsec: iotaS Enriques}

We now focus on the case where $\iota_S$ is  an Enriques involution of $S$ (which is by definition a non symplectic involution such that the fixed locus of $\iota_S$ is empty). The quotient surface $W\simeq S/\iota_S$ is an Enriques surface (a regular surface with the canonical bundle which is a 2-torsion bundle).

\begin{proposition}\label{prop:quoz_enriq}
If $\iota_S$ is an Enriques involution on $S$, then $Y_S$ is isomorphic to $\widetilde{Z_S/\sigma_Z}$ and they are the blow up of the non ramified double cover of $W^{(2)}$ in its singular locus.
\end{proposition}
\proof Both $Y_S$ and $\widetilde{Z_S/\sigma_Z}$ are desingularizations of $S^{(2)}/\iota_S^{(2)}$. The canonical bundle of $W^{(2)}$ is a 2-torsion bundle which induces the $2:1$ cover $S^{(2)}/ \iota^{(2)}\ra W^{(2)}$, thus $S^{(2)}/\iota^{(2)}$ is singular along two surfaces (mapped to $\Sing(W^{(2)})$). These surfaces are $\beta_2(\alpha_1(\Delta_S))$ and $\beta_2(\alpha_1(\Gamma_S)$. Since $S^{(2)}/ \iota^{(2)}\ra W^{(2)}$ is an unramified cover (or since $\Fix_{\iota_S}(S)$ is empty), the surfaces  $\beta_2(\alpha_1(\Delta_S))$ and $\beta_2(\alpha_1(\Gamma_S)$ are disjoint.

In order to construct $Y_S$, one first constructs $S^{[2]}$ as blow up of $S^{(2)}$ in its singular locus, which is $\gamma_1(\Delta_S)$. Then one constructs the quotient $S^{[2]}/\iota_S^{[2]}$ and blow up its singular locus, which is mapped on $S^{(2)}/\iota_S^{(2)}$ to  $\gamma_1(\gamma_2(\Gamma_S))$. So $Y_S$ is isomorphic to $S^{(2)}/\iota_S^{(2)}$ blown up in the two disjoint surfaces $\gamma_1(\gamma_2(\Delta_S))\simeq \alpha_1(\beta_2(\Delta_S))$ and $\gamma_1(\gamma_2(\Gamma_S))\simeq \alpha_1(\beta_2(\Gamma_S))$.

In order to construct $\widetilde{Z_S/\sigma_Z}$, one first constructs the smooth quotient $(S\times S)/(\iota_S\times \iota_S)$, then one considers the quotient $\beta_2:(S\times S)/(\iota_S\times \iota_S)\ra S^{(2)}/\iota_2^{(2)}$ and finally one blows up the two singular surfaces $\alpha_1(\beta(\Delta_S))$ and $\alpha_1(\beta(\Gamma_S))$ of $S^{(2)}/\iota_2^{(2)}$. 

So both $Y_S$ and $\widetilde{Z_S/\sigma_Z}$ are the blows up of $S^{(2)}/\iota_S^{(2)}$ in its two disjoint singular surfaces.\endproof

We remark that $Y_S\simeq\widetilde{Z_S/\sigma_S}$ (blow up of $S^{(2)}/\iota_S^{(2)}$ in its singular locus) is exactly the universal cover of $W^{[2]}$ (blow up of $W^{(2)}$ in its singular locus), whose existence is proven in \cite[Theorem 3.1]{Oguiso-Schroer}.

\section{Projective models}\label{sec: projective models}
The aim of this section is to describe some explicit models of the 4-folds constructed relating the geometric description given by the diagram \eqref{eq: diagram quotient by D8} and by Proposition \ref{prop: description of 4-folds as double covers} with the description of the cohomology on $S^{[2]}$, $Z_S$ and $Y_S$, given Sections \ref{subsec: the picard  group of YS}, \ref{subsec: Y_S} and \ref{subsec: ZS}.

The main results of this Section are of two different types: first we consider divisors induced on the 4-folds $S^{[2]}$, $Z_S$ and $Y_S$ by nef and big divisors on $S$ and we compute their characteristic (Proposition \ref{chi-formula}) and in some cases the dimension of their linear system (Theorem \ref{prop: h^0 for the divisors H}). Thus we give an explicit formulation of Riemann-Roch theorem in our context. Then we apply these general results to specific examples of divisors and K3 surfaces $S$, in Sections \ref{subsec: S/iota=P2}, \ref{subsec: S/iota=P1P1} and \ref{subsec: S/iota=F4}. In some particular cases we also gives explicit equations for some of the 4-folds constructed, see Section \ref{subsubsec: S/ioa=P2, branch=line+quintic} and proof of Propositions \ref{prop: maps S/iota=P1P1} and \ref{prop: fibrations FX, S/iota=F4}.

If a K3 admits a non symplectic involution $\iota_S$, then $S$ is contained in at least one of these connected components of the moduli spaces of K3 surfaces: $S$ is a $\langle 2\rangle$-polarized K3 surface and in this case $S/\iota_S$ is generically isomorphic to $\mathbb{P}^2$; $S$ is a $U(2)$-polarized K3 surface and in this case $S/\iota_S$ is generically isomorphic to $\mathbb{P}^1\times \mathbb{P}^1$; $S$ is a $U$-polarized K3 surface and in this case $S/\iota_S$ is generically isomorphic to the Hirzebruch surface $\mathbb{F}_4$. We analyze these different situations in Sections \ref{subsec: S/iota=P2}, \ref{subsec: S/iota=P1P1} and \ref{subsec: S/iota=F4} respectively.

\subsection{Results on divisors and linear systems}

Assume now that there is the following commutative diagram:

$$
\xymatrix{ \tilde{E}\ar@{^{(}->}[r]\ar[d]_{\beta_{|\tilde{E}}} &\widetilde{V}\ar[d]_{\beta}\ar[r]_{\pi}&X\ar[d]_{\beta'}\\
\Sigma\ar@{^{(}->}[r]_{j}&V\ar[r]_{\pi'}&V/\alpha},
$$ 
with:
\begin{itemize}
\item $V$ a smooth 4-fold with $c_1(V)=0$ (in the sequel it will be either $V=S\times S,\ Z_S$ or $S^{[2]}$);
\item $\alpha\in\Aut(V)$ a period preserving involution (in the sequel it will be either $\alpha=\sigma,\ \sigma_Z$ or $\iota_S\times \iota_S$);
\item $\Sigma$ the smooth surface fixed by $\alpha$ in $V$, embedded via $j$;
\item $\beta:\widetilde{V}\ra V$  the blow-up of $V$ along $\Sigma$; 
\item $\tilde{E}$ the exceptional divisor over $\Sigma$;
\item $\pi':V\ra V/\alpha$ the quotient map; $V/\alpha$ is singular in $\pi'(\Sigma)$, 
\item $\beta':X\ra V/\alpha$ the blow up of $V/\alpha$ in its singular locus, $\pi'(\Sigma)$;
\item $X=\widetilde{V}/\widetilde{\alpha}$, where $\widetilde{\alpha}$ is the involution induced by $\alpha$ on $\tilde{V}$ and $\pi:\widetilde{V}\ra X$ is the quotient map;
\item  $c_1(X)=0$, as it will be the case in our application.
\end{itemize}
The map $\pi:\widetilde{V}\ra X$ is a double cover ramified along $\tilde{E}$ and branched along the exceptional divisor $E$ of the blow up $\beta':X\ra V/\alpha$. In particular $\pi(\tilde{E})\simeq E$.

Let $D$ be a divisor on $V$ invariant for $\alpha$; we set $\tilde{D}:=\beta^*D$ and denote by $D_{X}$ the divisor on $X$ such that $\pi^*D_{X}=\tilde{D}=\beta^*D$.

\begin{lemma}\label{lemma: divisors induced}
 If $D$ is big and nef on $V$, then $D_{X}$ is big and nef on $X$.
\end{lemma}
\proof
The ample (or big and nef) divisors on $V$ which are invariant for an automorphism $\alpha\in \Aut(V)$, descend to ample (or big and nef) divisors on the quotient $V/\alpha$. In particular, bigness of nef divisors is preserved under finite quotient maps by \cite[Prop. 2.61]{kollarbir}, as the sign of the top self intersection does not change. These divisors on the possibly singular quotient $V/\alpha$ induce nef divisors on $X$, and bigness is a birational invariant (see e.g. \cite[Defn. 2.59 and Lemma 2.60]{kollarbir}). 

\endproof

For a divisor $D$ on a fourfold $X$ with $c_1(X)=0$, the Riemann--Roch formula (see for example \cite[Corollary 15.2.1]{Fulton}) is 
\begin{equation}\label{RR4}
\chi(D)=\frac{D^4}{24}+\frac{1}{24}D^2.c_2(X)+\chi(\mathcal{O}_X)
\end{equation}
If $X$ is Calabi--Yau, $\chi(\mathcal{O}_X)=2$, and if $X=S\times S$, $\chi(\mathcal{O}_X)=4$. 

Moreover, for a divisor $D$ on a fourfold of $\hsk$ type,  the Riemann-Roch formula can be written in terms of the Beauville--Bogomolov--Fujiki quadratic form $q$ on $H^2(S^{[2]},\Z)$:
\[
\chi(D)=\frac{1}{8}(q(D)+4)(q(D)+6).
\]

We want now to compute $\chi(D_X)$ in terms of $\chi(D)$, and to do so we need to understand Chern classes of $X$ in terms of Chern classes of $V$.

\begin{proposition}\label{c1-c2(X)}
 Under the above assumptions, we have $$c_2(X)=\frac{1}{2}\pi_*\beta^*c_2(V)+\frac{1}{2}\pi_*\beta^*j_*[\Sigma]-\pi_*([\tilde{E}]^2).$$
\end{proposition}
\proof
 We apply the theory explained in \cite[\S 3.5]{Esnault-Viehweg} (see also \cite{Cynk-vanStraten}): note that $\pi$ is a double cover branched along the smooth divisor $E$. Then there is $L\in \Pic(X)$ such that $L^{\otimes 2}=\mathcal{O}(E)$ and $\pi_*T_{\widetilde{V}}=T_{X}\otimes L^{-1}\oplus T_{X}(\log E)$,
with
\[
\xymatrix{
0\ar[r]&T_{X}(\log E)\ar[r]&T_{X}\ar[r]&N_{E|X}\ar[r]&0
}\]

Since $c_1(X)=0$ we have $c_1(T_X(\log E))=-c_1(N_{E|X})=-c_1(\of_E(E))$. From the exact sequence
\[
\xymatrix{
0\ar[r]&\of_{X}\ar[r]&\of_X(E)\ar[r]&\of_E(E)\ar[r]&0
}\]
we deduce $c_1(\of_E(E))=[E]$, 
$c_2(\of_E(E))
=0$.

Hence, $c_1(T_X(\log E))=c_1(X)-c_1(N_{E|X})
=-E$
and $$c_2(T_X(\log E))=c_2(X)-c_2(N_{E|X})-c_1(N_{E|X})c_1(T_X(\log E))=c_2(X)+E^2.$$

Next, we use the following formula for the tensor product of a vector bundle and a line bundle: $\ch(V\otimes L)=\ch(V)*\ch(L)$, from which:

\[c_1(T_{X}\otimes L^{-1})=\rk T_X*c_1(L^{-1})+c_1(X)=-4c_1(L)=-2E,\]
\[c_2(T_{X}\otimes L^{-1})=c_2(X)+3c_1(X)c_1(L^{-1})+\binom{4}{2}c_1(L^{-1})^2
=c_2(X)+\frac{3}{2}E^2\]

And we thus obtain:
\[c_1(\pi_*T_{\widetilde{V}})=c_1(T_{X}\otimes L^{-1})+c_1(T_X(\log E))
=-3E,\]
\[c_2(\pi_*T_{\widetilde{V}})=c_2(T_{X}\otimes L^{-1})+c_2(T_X(\log E))+c_1(T_{X}\otimes L^{-1})c_1(T_X(\log E))=\]
\[
=2c_2(X)+\frac{9}{2}E^2.\]

Since $\pi$ is finite, $R^i\pi_*T_{\widetilde{V}}=0$ for $i>0$ and we can apply Grothendieck--Riemann--Roch theorem \cite[Theorem 15.2]{Fulton}: 
$$\ch(\pi_* T_{\widetilde{V}})\,\td(T_{X})=\pi_*(\ch(\widetilde{V})\,\td(T_{\widetilde{V}})),\ \mbox{ i.e.}$$ 
 \[
 [8-3E+\frac{1}{2}(-3E)^2-2c_2(X)-\frac{9}{2}E^2+\dots][1+\frac{1}{12}c_2(X)+\dots]=\]
 \[\pi_*[(4+c_1(\widetilde{V})+\frac{1}{2}c_1(\widetilde{V})^2-c_2(\widetilde{V})+\dots)(1+\frac{1}{2}c_1(\widetilde{V})+\frac{1}{12}c_1(\widetilde{V})^2+\frac{1}{12}c_2(\widetilde{V})+\dots)]
 \]
which yields the following formulas:
 $$\pi_*c_1(\widetilde{V})=-E,\ c_2(X)=-\pi_*(c_1(\widetilde{V})^2)+\frac{1}{2}\pi_*(c_2(\widetilde{V})).$$

Finally, we remark that, by \cite[Example 15.4.3]{Fulton},  
$c_2(\widetilde{V})=\beta^*c_2(V)+\beta^*j_*[\Sigma],$ and this ends the proof.
\endproof

\begin{proposition}\label{chi-formula}
 With the notation above, we have
 \begin{equation*}
  \chi(D_X)=\frac{1}{2}\chi(D)+\frac{1}{16}(D_{|\Sigma})^2-\frac{1}{2}\chi(\of_V)+\chi(\of_X).
 \end{equation*}

\end{proposition}
\proof
By Riemann--Roch (\ref{RR4}), $\chi(D_X)=\frac{1}{24}D_X^4+\frac{1}{24}D_X^2.c_2(X)+\chi(\mathcal{O}_X)$. 

By \cite[Proposition 1.10]{Debarre}, $\widetilde{D}^4=(\pi^*D_X)^4=2D_X^4$, and on the other hand $\widetilde{D}^4=D^4$ by \cite[find statement]{Fulton}, hence $D_X^4=\frac{1}{2}D^4$.

We need now to compute $D_X^2.c_2(X)$. We use Proposition \ref{c1-c2(X)}, projection formula \cite[Proposition 8.3(c)]{Fulton} and the fact that $\pi^*E=2\tilde{E}$ and $\pi_*\tilde{E}=E$, getting:
\[
D_X^2.\pi_*\beta^*c_2(V)=\pi_*(\tilde{D}^2.\beta^*c_2(V))= D^2.c_2(V),
\]
\[
D_X^2.\pi_*\beta^*j_*[\Sigma]=\pi_*(\tilde{D}^2.\beta^*j_*[\Sigma])=\pi_*(\beta^*(D^2.j_*[\Sigma]))=\pi_*(\beta^*((D_{|\Sigma})^2)),
\]
\[
D^2_X.\pi_*([\tilde{E}]^2)=D^2_X.\pi_*(\frac{1}{2}\pi^*E.\tilde{E})=\frac{1}{2}D^2_X.E^2=\frac{1}{2}\pi_*(\tilde{D}^2.\tilde{E}^2)=-\frac{1}{2}\pi_*(\beta^*((D_{|\Sigma})^2)).
\]
where last equality follows from \cite[Lemma 1.1]{Badescu-Beltrametti}.
We plug everything into (\ref{RR4}) and obtain:
\[
\chi(D_X)=\frac{1}{48}D^4+\frac{1}{48}D^2.c_2(V)+\frac{1}{16}(D_{|\Sigma})^2+\chi(\of_X)
.
\]
\endproof

Let $H_S$ be a divisor on $S$.  In the sequel we will use the following notations:
\begin{itemize}
 \item $H_{1,S\times S}$ (resp. $H_{2,S\times S}$) is the divisor $H_S\otimes H^{0,0}(S)\in \NS(S\times S)$ (resp. $H^{0,0}(S)\otimes H_S\in \NS(S\times S)$);
 \item $H_{X}$ is the divisor respectively on $X=S^{[2]},\ Z_S$ such that $\pi^*(H_X)=\beta^*(H_{1,S\times S}+H_{2,S\times S})\subset NS(\widetilde{ S\times S})$;
 \item $H_{Y}$ is the divisor on $Y_S$ such that $\pi^*(H_Y)=\beta^*H_{S^{[2]}}\subset NS(\widetilde{S^{[2]}})$;
 \item $H_{1,Z}$, $H_{2,Z}$ are the divisors on $Z_S$ such that $\pi^*(H_{i,Z})=\beta^*(H_i,S\times S)\subset NS(\widetilde{S\times S})$.
 \end{itemize}

A straightforward application of Lemma \ref{lemma: divisors induced} yields the following:
\begin{corollary}\label{cor: divisors induced} Let $H_S$ be an ample (or big and nef) divisor on $S$.  The divisor $H_{1,S\times S}+H_{2,S\times S}$ is ample (or big and nef) on $S\times S$, and $H_{S^{[2]}}$ is big and nef on $S^{[2]}$. 
  
If moreover we assume that $H_S\in \NS(S)$ is invariant for $\iota_S$, then  $H_{Z}$ and $H_Y$ are big and nef divisors.
 \end{corollary}  

To avoid confusion, in the following Proposition we will denote the self intersection product between divisors by $(H\cdot H)$ (and not by $H^2$). Moreover we will use the following notation:
$h_{\Sigma_{S\times S}}$ (resp. $h_{\Sigma_{1, S\times S}}$, $h_{\Sigma_Z}$) denote the number of points in the intersection of the surface $\Sigma_{S\times S}\simeq \Fix_{\iota_S}(S)\times \Fix_{\iota_S}(S)\subset S\times S$ (resp. $\Sigma_{S\times S}\simeq \Fix_{\iota_S}(S)\times \Fix_{\iota_S}(S)\subset S\times S$, $\Sigma_{Z}\simeq \Fix_{\sigma_Z}(Z)\subset Z$) and a surface equivalent to  $\left(H_{1,S\times S}+H_{2,S\times S}\right)^2$ (resp. $H_{1,S\times S}^2$, $H_Z^2$).
\begin{theorem}\label{prop: h^0 for the divisors H} Let $H$ be a nef and big $\iota_S$-invariant divisor on $S$. 
Then 
\begin{align*}
\begin{array}{lll}
h^0(H_{S^{[2]}})&=&\frac{1}{8}\left((H\cdot H)+4\right)\left((H\cdot H)+6\right),\\
h^0(H_{Z})&=&\frac{1}{2}\left(h^0(H)\right)^2+\frac{1}{16}h_{\Sigma_{S\times S}},\\
h^0(H_{Y})&=&\frac{1}{2}h^0(H_Z)+\frac{1}{16} h_{\Sigma_{Z}}+1=\frac{1}{4}\left(h^0(H)\right)^2+\frac{1}{32}h_{\Sigma_{S\times S}}+\frac{1}{16}h_{\Sigma_{Z}}+1.
\end{array}
\end{align*}\end{theorem}

\proof Since $H$ is big and nef, by Kawamata--Viehweg vanishing theorem $\chi(H)=h^0(H)$. Similarly, by Corollary \ref{cor: divisors induced}, $\chi(H_X)=h^0(H_X)$ for $X=S^{[2]}, Z_S, Y_S$ since these divisors are are big and nef. Now the proposition is a trivial application of Proposition \ref{chi-formula}.\endproof
\begin{rem}{\rm 
The divisor $H_{1,Z}$ is not necessarily big and nef, thus we can not assume that $\chi(H_{1,Z})=h^0(H_{1,Z})$. However, one can compute $\chi(H_{1,Z})$ as one computes $h^0$ in Theorem \ref{prop: h^0 for the divisors H} for the divisors $H_Z$, $H_{S^{[2]}}$, $H_Y$. Thus one obtains $\chi(H_{1,Z})=\frac{1}{2}\chi(H)+\frac{1}{16} h_{\Sigma_{1,S\times S}}$.}\end{rem}
\begin{rem}\label{rem: divisors H orthogonal to exceptional}{\rm  The map induced by the linear systems of $H_X$, for $X=S^{[2]}, Z_S, Y_S$ and $H_{1,Z}$ contracts the exceptional divisors introduced by the blow up $\beta$ on these varieties. So all of them factorizes through the singular models described in diagram \eqref{eq: diagram quotient by D8}. In particular $\varphi_{|H_{S^{[2]}}|}$ defines a map on $S^{(2)}$, $\varphi_{|H_Z|}$ and $\varphi_{|H_{1,Z}|}$ define maps on $(S\times S)/(\iota_S\times\iota_S)$ and $\varphi_{|H_Y|}$ defines a map on $(S^{(2)})/(\iota_{S}^{(2)})\simeq (S\times S)/\langle \sigma, \iota_S\times \iota_S\rangle$. The target spaces of the map $\varphi_{|H_X|}$, for $X=S^{[2]}$ (resp. $Z_S$, $Y_S$) is a copy of $W^{(2)}$ (resp. $W\times W$, $W^{(2)}$) embedded in a projective space.
}\end{rem}

\subsection{$\iota_S$ is the cover involution of the $2:1$ map $S\ra\mathbb{P}^2$}\label{subsec: S/iota=P2}
Let us now assume that $\iota_S$ fixes on $S$ one curve of genus 10. In this case, generically, $\NS(S)\simeq \langle 2\rangle\simeq \Z H$ and $\varphi_{|H|}:S\ra \mathbb{P}^2$ is a $2:1$ cover branched along a smooth plane sextic denoted by $B$. So $H$ is an ample divisor, $h^0(H)=3$, $W\simeq \mathbb{P}^2$, $C$ is a genus 10 curve, isomorphic to a plane sextic and the class of $C$ in $\NS(S)$ is $3H$. 

The automorphisms group $\Aut(S\times S)$ is $\mathcal{D}_8$ and the admissible quotients are described in Proposition 
 \ref{prop: description of 4-folds as double covers}.

Let us denote by $i_2:\mathbb{P}^2\times \mathbb{P}^2\ra \mathbb{P}^5$ the map $((x_0:x_1:x_2),(y_0:y_1:y_2))\mapsto (x_0y_0:x_0y_1+x_1y_0:x_0y_2+x_2y_1:x_1y_1:x_1y_2+x_2y_1:x_2y_2)$ which exhibits $(\mathbb{P}^2)^{(2)}$ as subvariety of $\mathbb{P}^5$. 
\begin{proposition}\label{prop: maps if NS=2}
The map $\varphi_{|H_{S^{[2]}}|}:S^{[2]}\ra \mathbb{P}^5$ is a $2^2:1$ cover of $i_2((\mathbb{P}^2)^{(2)})\subset\mathbb{P}^5$ branched on the image of $i_2(B\times B)$. It contracts the exceptional divisors, thus it factorizes through $S^{(2)}$ inducing a $4:1$ map $S^{(2)}\ra W^{(2)}$ (cf. diagram (\ref{eq: diagram quotient by D8})).
 
The map $\varphi_{|H_{Z}|}:Z_S\ra \mathbb{P}^8$ is a $2:1$ map onto $\mathbb{P}^2\times\mathbb{P}^2$ embedded in $\mathbb{P}^8$ by the Segre embedding. The branch of $\varphi_{|H_{Z}|}:Z_S\ra \mathbb{P}^2\times\mathbb{P}^2\subset \mathbb{P}^8$ is the image of $B\times B$ by the Segre embedding. This map contracts the exceptional divisors, thus it factorizes through $(S\times S)/(\iota_S\times \iota_S)$ inducing the $2:1$ map $(S\times S)/(\iota_S\times \iota_S)\ra (W\times W$) (cf. diagram (\ref{eq: diagram quotient by D8})).

The map $\varphi_{|H_Y|}:Y_S\ra \mathbb{P}^5$ is a $2:1$ map to $(\mathbb{P}^2)^{(2)}$, embedded in $\mathbb{P}^5$ by $i_2$. The branch locus of $\varphi_{|H_{Y}|}$ is $i_2(B\times B)$. This map contracts the exceptional divisors, thus it factorizes through $(S\times S)/\langle\iota_S\times \iota_S,\sigma\rangle$ inducing the $2:1$ map $(S\times S)/\langle\iota_S\times \iota_S,\sigma\rangle\ra W^{(2)}$ (cf. diagram (\ref{eq: diagram quotient by D8})).\end{proposition}
\proof By Remark \ref{rem: divisors H orthogonal to exceptional} one obtains that the map described contracts the exceptional divisors and are defined on the singular models of the 4-folds considered. 
The number $h^0(H_{S^{[2]}})$ can be computed by Theorem \ref{prop: h^0 for the divisors H}. The image of $\varphi_{|H_{S^{[2]}}|}$ is $(\mathbb{P}^2)^{(2)}$ by the commutativity of the diagram
$$\xymatrix{ S\times S\ar[r]^-{\varphi_{|H|}\times \varphi_{|H|}}_{4:1}\ar[d]_{2:1}&\mathbb{P}^2\times \mathbb{P}^2\ar[d]^{2:1}\\
S^{(2)}\ar[r]^{\varphi_{\left|H_{S^{[2]}}\right|}}_{4:1}&(\mathbb{P}^2)^{(2)}}.$$
In order to compute $h^0(H_{Z})$ one has to recall that the class of $\Fix_{\iota_S}(S)\times \Fix_{\iota_S}(S)\subset S\times S$ is $(3H_{1,S})(3H_{2,S})$, so $h_{\Sigma_{S\times S}}=(H_{1,S}+H_{2,S})^2(3H_{1,S})(3H_{2,S})=9(H_{1,S}^3H_{2,S}+H_{2,S}^3H_{1,S}+2H_{1,S}^2H_{2,S}^2)$. Since $H_{i,S}$ is the pull-back of a divisor on $S$, $H_{i,S}^3=0$ and $H_{1,S}^2H_{2,S}^2=2^2$ since $H^2=2$ on $S$. Thus  $h_{\Sigma_{S\times S}}=9\cdot 2 \cdot 4=72$.
By Theorem \ref{prop: h^0 for the divisors H}, $h^0(H_Z)=\frac{9}{2}+\frac{72}{16}=9$. The image of $\varphi_{|H_Z|}$ is $\mathbb{P}^2\times\mathbb{P}^2$ by the commutativity of the diagram
$$\xymatrix{ S\times S\ar[r]^{\varphi_{|H|}\times \varphi_{|H|}}_{4:1}\ar[d]_{2:1}&\mathbb{P}^2\times \mathbb{P}^2\\
(S\times S)/(\iota_S\times\iota_S)\ar[ur]_{\varphi_{\left|H_Z\right|}}^{2:1}.}$$  
In order to compute $h^0(H_{Y})$ one has to recall that $\sigma_Z$ fixes on $Z$ the image of $\Delta_S\subset S\times S$ and the image of $\Gamma_S\subset S\times S$. First we observe that on $S\times S$ it holds $(H_{1,S}+H_{2,S})^2\Delta_S=(H_{1,S}+H_{2,S})^2\Gamma_S=4(H\cdot H)$, which in our case implies $(H_{1,S}+H_{2,S})^2\Delta_S=(H_{1,S}+H_{2,S})^2\Gamma_S=8$. Then we observe that, since $H$ is a movable divisor, generically the points in $(H_{1,S}+H_{2,S})^2\Delta_S$ and $(H_{1,S}+H_{2,S})^2\Gamma_S$ are not contained in the fixed locus of $\iota_s\times \iota_S$, the involution $\iota_S\times\iota_S$ is non trivial on these points and the quotient by $\iota_S\times \iota_S$ identifies pairs of these points. Hence the intersection between $H_Z$ and the image of $\Delta_S$ (resp. $\Gamma_S$) in $Z$ consists of 4 points and thus $h_{\Sigma_Z}=4+4=8$.
By Theorem \ref{prop: h^0 for the divisors H}, $h^0(H_Y)=\frac{9}{2}+\frac{8}{16}+1=6.$ 
The image of $\varphi_{|H_{Y}|}$ is $(\mathbb{P}^2)^{(2)}$ by the commutativity of the diagram:
$$\xymatrix{ S\times S\ar[r]^{\varphi_{|H|}\times \varphi_{|H|}}_{4:1}\ar[d]_{2:1}&\mathbb{P}^2\times \mathbb{P}^2\ar[r]_{2:1}&(\mathbb{P}^2)^{(2)}\ar@{^{(}->}[r]^{i}&\mathbb{P}^5\\
(S\times S)/\langle\iota_S\times\iota_S\rangle\ar[r]_{2:1}&(S\times S)/\langle \iota_S\times \iota_S,\sigma\rangle\ar[ur]_{2:1}&}.$$\endproof

\begin{rem}{\rm The map $\varphi_{H_{1,Z}}$ contracts the exceptional divisor and induces on $(S\times S)/(\iota_S\times \iota_S)$ the map $(S\times S)/(\iota_S\times \iota_S)\ra W\simeq \mathbb{P}^2$ whose generic fibers are isomorphic to $S$. Since $H_{1,Z}$ is not a big divisor, we do not know a priori that $\chi(H_{1,Z})=h^0(H_{1,Z})$. In any case we observe that the geometric description of the map $\varphi_{|H_{1,Z}|}$ suggests that $h^0(H_{1,Z})$ should be 3 and applying Remark \ref{rem: divisors H orthogonal to exceptional} one finds $\chi (H_{1,Z})=3$. Indeed by K\"unneth formula one finds $h^{0}(H_{1,S})=3$, $h^i(H_{1,S})=0$, $i=1,3,4$ and $h^2(H_{1,S})=3$, and since $h_{\Sigma_{1, S\times S}}=0$, one obtains $\chi(H_{1,Z})=\chi(H_{1,S})/2=6/2$.}\end{rem}

\subsubsection{A special case: $S$ is the double cover of $\mathbb{P}^2$ branched on a line and a (possibly reducible) quintic.}\label{subsubsec: S/ioa=P2, branch=line+quintic}
In this case the hypotheses of Proposition \ref{prop: description of 4-folds as double covers} are not satisfied, since the fixed locus of $\iota_S$ contains at least 2 curves. On the other hand there is a nice birational description of the Calabi--Yau involved in our construction. 
So let us assume that $S$ is the double cover of $\mathbb{P}^2$ branched along a line and a quintic. In this case $W$ is a blow up of $\mathbb{P}^2$, and if the quintic is smooth and intersects the line transversally, it is a blow up of $\mathbb{P}^2$ in 5 (collinear) points. We denote by $p_i:S\times S\ra S$ the $i$-th projection. An equation of a birational (singular) model of $S\simeq p_1(S\times S)$ is given by $$X^2=x_0f_5(x_0:x_1:x_2)$$
which exhibits $S$ as double cover of $\mathbb{P}^2_{(x_0:x_1:x_2)}$. Let us denote by $Y^2=y_0f_5(y_0:y_1:y_2)$ the analogous equation for $S\simeq p_2(S\times S)$. The action of $\iota_S\times \iota_S$ is given by $(X,(x_0:x_1:x_2);Y,(y_0:y_1:y_2))\ra (-X,(x_0:x_1:x_2);-Y,(y_0:y_1:y_2))$.

We now consider the affine equation of $p_1(S\times S)$ and $p_2(S\times S)$ obtained by putting $x_0=1$ and $y_0=1$. The invariant functions for $\iota_S\times \iota_S$ are $Z:=XY$, $a_1=x_1$, $a_2=x_2$, $a_3=y_1$, $a_4=y_2$. Then a birational equation for $(S\times S)/(\iota_S\times \iota_S)$ (and thus a birational model of $Z_S$), is given by 
$$Z^2=f_5(1:a_1:a_2)f_5(1:a_3:a_4).$$
This equation exhibits $Z_S$ as double cover of the affine subspace of $\mathbb{P}^4_{(a_0:a_1:a_2:a_3:a_4)}$ given by $a_0=1$. It is clearly possible to introduce the variable $a_0$ in order to obtain a homogenous polynomial $F_{10}(a_0:a_1:a_2:a_3:a_4)$ of degree 10  which reduces to $f_5(1:a_1:a_2)f_5(1:a_3:a_4)$ if $a_0=1$. So $Z_S$ is birational to a $2:1$ cover of $\mathbb{P}^4$ branched along a (possibly singular) $3$-fold of degree 10, denoted by $B$.  

On $\mathbb{P}^4$, the map $\sigma_{\mathbb{P}^4}:(a_0:a_1:a_2:a_3:a_4)\mapsto (a_0:a_3:a_4:a_1:a_2)$ acts preserving the homogeneous polynomial $F_{10}(a_0:a_1:a_2:a_3:a_4)$. The map $\sigma_Z$ is induced on $Z_S$ by the projective map $\sigma_{\mathbb{P}^4}$. Denoted by $\pi$ the quotient map $\pi:\mathbb{P}^4\ra\mathbb{P}^4/\sigma_{\mathbb{P}^4}$, we obtain that $Z_S/\sigma_Z$ and $Y_S$ are birational to a double cover of $\pi(\mathbb{P}^4)$ branched over $\pi(V(F_{10}(a_0:a_1:a_2:a_3:a_4)))$ (where $V(F_{10}(a_0:a_1:a_2:a_3:a_4))$ is the zero locus of the polynomial $F_{10}$).

In order to better describe $\mathbb{P}^4/\sigma_{\mathbb{P}^4}$ it is convenient to apply the changes of coordinates $b_0:=a_0$, $b_1:=(a_1+a_2)/2$, $b_2:=(a_3+a_4)/2$, $b_3:=(a_1-a_2)/2$, $b_4:=(a_3-a_4)/2$. With these new coordinates $\sigma_{\mathbb{P}^4}$ is the map  $\sigma_{\mathbb{P}^4}:(b_0:b_1:b_2:b_3:b_4)\mapsto (b_0:b_1:b_2:-b_3:-b_4)$, and $\mathbb{P}^4/\sigma_{\mathbb{P}^4}$ is mapped to the $4$-dimensional singular subspace of $\mathbb{P}^8_{(z_0:\ldots:z_8)}$ given by the set-theoretic complete intersection of 4 singular quadrics $M:=V(z_0z_1=z_5^2,\ z_0z_2=z_6^2,\ z_1z_2=z_7^2,\ z_3z_4=z_8^2)$, where $z_i:=b_i^2$ for $i=0,\ldots, 4$, $z_5:=b_0b_1$, $z_6:=b_0b_2$, $z_7:=b_1b_2$, $z_8:=b_3b_4$. The space $M$ is singular, $\Sing(M)=\left(M\cap V(z_0,z_1,z_2,z_5,z_6,z_7)\right)\cup\left(M\cap V(z_3,z_4,z_8)\right)$.

The image of $V(F_{10})$ under the quotient map is a 3-fold $T$ of degree 5 in $\mathbb{P}^8$,  Then $Y_S$ is birational to a double cover of $M$ branched along  $M\cap T$.

\subsection{$\iota_S$ is the cover involution of the $2:1$ map $S\ra\mathbb{P}^1\times \mathbb{P}^1$}\label{subsec: S/iota=P1P1}

Let us now assume that $\iota_S$ fixes on $S$ one curve of genus 9 isomorphic to the branch curve $B$, which has bidegree $(4,4)$ in $\mathbb{P}^1\times \mathbb{P}^1$. In this case, generically, $\NS(S)\simeq U(2)\simeq \Z l\oplus \Z m$ and $\varphi_{|l+m|}:S\ra \mathbb{P}^1\times \mathbb{P}^1\subset\mathbb{P}^3$ is a $2:1$ map on the image and $\mathbb{P}^1\times \mathbb{P}^1$ is embedded in $\mathbb{P}^3$ by the Segre embedding. We denote $H:=l+m\in \NS(S)$. The branch divisor of the $2:1$ cover $S\ra\mathbb{P}^1\times \mathbb{P}^1\subset \mathbb{P}^3$ is thus represented by $2H$. We observe that $W\simeq \mathbb{P}^1\times \mathbb{P}^1$.

The automorphisms group $\Aut(S\times S)$ is $\mathcal{D}_8$ and the admissible quotients are described in Proposition \ref{prop: description of 4-folds as double covers}.

Let us denote by $i_3:\mathbb{P}^3\times \mathbb{P}^3\ra \mathbb{P}^9$ the map $((x_0:x_1:x_2:x_3),(y_0:y_1:y_2:y_3))\mapsto (x_0y_0 :x_0y_1+x_1y_0: x_0y_2+x_2y_1: x_0y_3+x_3y_0: x_1y_1: x_1y_2+x_2y_1: x_1y_3+x_3y_1: x_2y_2: x_2y_3+x_3y_2: x_3y_3)$ which exhibits $(\mathbb{P}^3)^{(2)}$ as a subvariety of $\mathbb{P}^9$. 
\begin{proposition}\label{prop: maps if NS=U(2)}
The map $\varphi_{|H_{S^{[2]}}|}:S^{[2]}\ra \mathbb{P}^9$ is a $2^2:1$ cover of $i_3\left(\left(\mathbb{P}^3\right)^{(2)}\right)\subset\mathbb{P}^5$ branched on the image of $i_3(B\times B)$. It contracts the exceptional divisors, thus it factorizes through $S^{(2)}$ inducing a $4:1$ map $S^{(2)}\ra W^{(2)}$ (cf. diagram \eqref{eq: diagram quotient by D8}).
 
The map $\varphi_{|H_{Z}|}:Z_S\ra \mathbb{P}^{15}$ is a $2:1$ map onto $\mathbb{P}^1\times\mathbb{P}^1\times \mathbb{P}^1\times \mathbb{P}^1\subset \mathbb{P}^3\times \mathbb{P}^3$ embedded in $\mathbb{P}^{15}$ by the Segre embedding. The branch of $\varphi_{|H_{Z}|}:Z_S\ra \mathbb{P}^1\times\mathbb{P}^1\times \mathbb{P}^1\times\mathbb{P}^1\subset \mathbb{P}^{15}$ is the image of $B\times B $ by the Segre embedding. This map contracts the exceptional divisors, thus it factorizes through $(S\times S)/(\iota_S\times \iota_S)$ inducing the $2:1$ map $(S\times S)/(\iota_S\times \iota_S)\ra (W\times W$) (cf. diagram \eqref{eq: diagram quotient by D8}).

The map $\varphi_{|H_Y|}:Y_S\ra \mathbb{P}^{9}$ is a $2:1$ map to $(\mathbb{P}^3)^{(2)}$, embedded in $\mathbb{P}^9$ by $i_3$. The branch locus is $i_3(B\times B)$.This map contracts the exceptional divisors, thus it factorizes through $(S\times S)/\langle\iota_S\times \iota_S,\sigma\rangle$ inducing the $2:1$ map $(S\times S)/\langle\iota_S\times \iota_S,\sigma\rangle\ra W^{(2)}$ (cf. diagram \eqref{eq: diagram quotient by D8}).

\end{proposition}
The proof is analogous to the one of Proposition \ref{prop: maps if NS=2} and so we omit it. On the other hand in this case some explicit equations can be written and these can be used to describe some maps, associated to divisors which are not necessarily big and nef. 
 
The notation will be analogous to the one introduced above, where we substitute $H$ by $l$ or $m$. We observe that $l$ and $m$ are not big divisors on $S$, thus $l_X$, $m_X$ for $X=S\times S, Z_S, Y_S. S^{[2]}$ are not necessarily big divisors. 
 
An equation for the surface $S$, which exhibits $S$ as double cover of $\mathbb{P}^1_{(x_0:x_1)}\times \mathbb{P}^1_{(x_2:x_3)}$ is $X^2=f_{4,4}((x_0:x_1):(x_2:x_3))$, where $f_{4,4}$ is a homogeneous polynomials of bidegree $(4,4)$ in $\mathbb{P}^1\times \mathbb{P}^1$. Composing the map $S\ra\mathbb{P}^1\times \mathbb{P}^1$ with the projection of $\mathbb{P}^1\times \mathbb{P}^1$ on the first factor, we obtain a map $S\ra \mathbb{P}_{(x_0:x_1)}^1$, which is a genus 1 fibration. The fibers over a point $(\overline{x_0}:\overline{x_1})$ is the genus one curve $X^2=f_{4,4}((\overline{x_0}:\overline{x_1}):(x_2:x_3))$. The map $S\ra \mathbb{P}_{(x_0:x_1)}^1$ coincides with the map $\varphi_{|l_2|}$. Similarly one obtains another genus 1 fibration, projecting $\mathbb{P}^1\times \mathbb{P}^1$ on the second factor.
Here we describe some models and fibrations on $S^{[2]}$, $Z_S$ and $Y_S$ induced by the maps $S\ra\mathbb{P}^1\times\mathbb{P}^1$ and $S\ra\mathbb{P}^1$. In particular we will prove the following 

\begin{proposition}\label{prop: maps S/iota=P1P1}
The \hk 4-fold $S^{[2]}$ admits a Lagrangian fibration $f_{S^{[2]}}=\varphi_{|l_{S^{[2]}}|}:S^{[2]}\ra\mathbb{P}^2\simeq (\mathbb{P}^1)^{(2)}$ whose generic fibers are the product of two non isogenous elliptic curves.

The Calabi--Yau 4-fold $Z_S$ admits:\begin{itemize}\item a fibration $\varphi_{|l_{1,Z}|}:Z_S\ra\mathbb{P}^1$ whose generic fibers are Calabi--Yau 3-folds which are the double cover of $\mathbb{P}^1\times \mathbb{P}^1\times\mathbb{P}^1$ branched along the union of 5 curves of tridegree $(1,0,0)$,$(1,0,0)$,$(1,0,0)$,$(1,0,0)$,$(0,4,4)$;
\item a fibration $\varphi_{|l_{1,Z}+m_{1,Z}|}:Z_S\ra \mathbb{P}^1\times \mathbb{P}^1$ whose generic fibers are isomorphic to $S$;
\item a fibration $f_Z:=\varphi_{|l_Z|}:Z_S\ra\mathbb{P}^1\times\mathbb{P}^1$ whose generic fibers are the Kummer surfaces of the product of two non-isogenous elliptic curves;
\item a fibration $\varphi_{|l_{Z}+m_{1,Z}|}:Z_S\ra\mathbb{P}^1\times \mathbb{P}^1\times \mathbb{P}^1$ such that the generic fibers are smooth irreducible curves of genus 1.
\end{itemize}
The Calabi--Yau 4-fold $Y_S$ admits a fibrations $f_Y:=\varphi_{|l_Y|}:Y_S\ra (\mathbb{P}^1)^{(2)}\simeq \mathbb{P}^2$ whose fibers are the Kummer surfaces of the product of two non-isogenous elliptic curves. The fibration $f_Y$ is induced on $Y_S$ both by $f_{S^{[2]}}$ and by $f_Z$.
\end{proposition}
\proof
The map $\varphi_{|l_{S^{[2]}}|}:S^{2}\ra \mathbb{P}^2\simeq \left(\mathbb{P}^1\right)^{(2)}$ gives Lagrangian fibrations on $S^{[2]}$ whose fibers are the product of the fibers of $\varphi_{|l_i|}:S\ra\mathbb{P}^1$ of each factor in $S\times S$. We denote this fibration by $f_{S^{[2]}}$, following \cite[Example 3.5]{sawon03}. The involution $\iota_{S^{[2]}}$ acts on the fibers of  $\varphi_{|l_{S^{[2]}}|}:S^{2}\ra \mathbb{P}^2\simeq \left(\mathbb{P}^1\right)^{(2)}$ preserving each fibers, so it acts as an involution on each fiber. On the other hand we know that the fixed locus of $\iota_S^{[2]}$ consists of two surfaces, one isomorphic to $(\Fix_{\iota_S}(S))^{[2]}$ and one isomorphic to $\mathbb{P}^1\times\mathbb{P}^1$. The surface $(\Fix_{\iota_S}(S))^{[2]}$ intersects the fiber in 16 points and indeed $\iota_S^{[2]}$ restricts to each fiber to the involution which sends each point of an Abelian surface in its opposite. The surface isomorphic to $\mathbb{P}^1\times\mathbb{P}^1$ in 
the fixed locus of $\iota_{S}^{[2]}$ maps to the singular locus of $(\mathbb{P}^1\times\mathbb{P}^1)^{(2)}$ so it does not intersect the generic fiber. Hence the fibration $\varphi_{|l{i,S^{[2]}}|}:S^{2}\ra \mathbb{P}^2\simeq \left(\mathbb{P}^1\right)^{(2)}$ induces on $S^{[2]}/\iota_S^{[2]}$ a fibration $f:S^{[2]}/\iota_S^{[2]} \ra\mathbb{P}^2\simeq \left(\mathbb{P}^1\right)^{(2)}$ whose generic fibers are Kummer surfaces. This fibration extends to a map $f_Y: Y_S\ra\mathbb{P}^2\simeq \left(\mathbb{P}^1\right)^{(2)}$ whose generic fibers are Kummer surfaces.
By construction the map $f_Y$ is induced on $Y$ by the divisor $l_Y$, since $f_{S^{[2]}}$ is induced on $S^{[2]}$ by the divisor $l_{S^{[2]}}$.
  
Let us now consider $Z_S$. In order to describe the map in the statement we first give an equation for the surface $S$, which exhibits $S$ as double cover of $\mathbb{P}^1_{(x_0:x_1)}\times \mathbb{P}^1_{(x_2:x_3)}$: $$X^2=f_{4,4}((x_0:x_1):(x_2:x_3)),$$ where $f_{4,4}$ is a homogeneous polynomials of bidegree $(4,4)$ in $\mathbb{P}^1\times \mathbb{P}^1$. 
The second copy of $S$ in the product $S\times S$ is given by the equation $Y^2=f_{4,4}((y_0:y_1):(y_2:y_3)),$ which exhibits it as double cover of $\mathbb{P}^1\times \mathbb{P}^1$.

The divisor $l_{Z}+m_{Z}$ is  $H_{Z}$ and we already observed that $\varphi_{|H_Z|}:Z_S\ra \mathbb{P}^1\times \mathbb{P}^1\times \mathbb{P}^1\times\mathbb{P}^1\subset\mathbb{P}^{15}$. It exhibits $Z_S$ as double cover of $\mathbb{P}^1\times \mathbb{P}^1\times \mathbb{P}^1\times\mathbb{P}^1$ branched along a threefold of multidegree $(4,4,4,4)$ (by adjunction this is indeed a 4-fold with a trivial canonical bundle). 
The involution $\iota_S\times \iota_S$ acts only on the coordinates $X$ and $Y$, changing the sign so we choose as invariant functions $Z:=XY$, $x_i$ and $y_i$, $i=0,\ldots, 4$.
The equation of $Z_S$ is then 
\begin{equation}\label{eq: ZS double cover P1,4}Z^2=f_{4,4}((x_0:x_1):(x_2:x_3))f_{4,4}((y_0:y_1):(y_2:y_3)).\end{equation}

If we project \eqref{eq: ZS double cover  P1,4} to the first three copies of $\mathbb{P}^1$, one obtains a fibration $Z_S\ra \mathbb{P}^1_{(x_0:x_1)}\times \mathbb{P}^1_{(x_2:x_3)}\times\mathbb{P}^1_{(y_0:y_1)}$ whose generic fibers are the genus 1 curves $Z^2=kf_{4,4}((\overline{y_0}:\overline{y_1}):(y_2:y_3))$, where $\overline{y_i}$ are specific value for $y_i$ and $k$ is a constant which depends on the values of $x_i$. This fibration is induced on $Z_S$ by the map $\varphi_{|l_{Z}+m_{1,Z}|}:Z_S\ra\mathbb{P}^1\times \mathbb{P}^1\times \mathbb{P}^1$.

If we project \eqref{eq: ZS double cover  P1,4} to the first two copies of $\mathbb{P}^1$ one obtains a fibration $Z_S\ra \mathbb{P}^1_{(x_0:x_1)}\times \mathbb{P}^1_{(x_2:x_3)}$ whose generic fibers are isomorphic to $S$ (indeed they are double covers of $\mathbb{P}^1\times \mathbb{P}^1$ branched along the curve of bidegree $(4,4)$ given by $f_{4,4}((y_0:y_1)(y_2:y_3))$). This fibration is induced on $Z_S$ by the map $\varphi_{|l_{1,Z}+m_{1,Z}|}:Z_S\ra \mathbb{P}^1\times \mathbb{P}^1$. By definition, $l_{1,Z}+m_{1,Z}=H_{1,Z}$.

If we project \eqref{eq: ZS double cover  P1,4} to the first and to the third copy of $\mathbb{P}^1$, one obtains a fibration $f_Z:Z_S\ra \mathbb{P}^1_{(x_0:x_1)}\times \mathbb{P}^1_{(y_0:y_1)}$ whose generic fibers are K3 surfaces, not isomorphic to $S$. The fibers over a generic point $((\overline{x_0}:\overline{x_1}):(\overline{y_0}:\overline{y_1}))$, are the double covers of $\mathbb{P}^1_{(x_2:x_3)}\times \mathbb{P}^1_{(y_2:y_3)}$ branched along the curve of bidegree $(4,4)$ $f_4((\overline{x_0}:\overline{x_1}):(x_2:x_3))f_{4}((\overline{y_0}:\overline{y_1}):(x_2:x_3))$. But this curve splits in the union of 8 curves, 4 of bidegree  $(1,0)$ and 4 of bidegree $(0,1)$. So the branch locus of this double cover is singular in 16 points and thus the K3 surfaces obtained by blowing up these points contain 16 disjoint rational curves. This suffices to conclude that each fiber of the fibration $f_Z$ is a Kummer surface (see \cite{Nikulin-Kummers}). To be more precise, the generic fiber is a K3 
surface which contains 24 rational curves (the pull back of the eight curves in the branch locus of the double cover of $\mathbb{P}^1\times \mathbb{P}^1$ and the 16 curves which resolve the singularities) which form a double Kummer configuration (see \cite{Oguiso-Jacobian}). We conclude that the generic fibers are Kummer surfaces of the product of two non isogenous elliptic curves. The fibration $f_Z:Z_S:\ra\mathbb{P}^1_{(x_0:x_1)}\times \mathbb{P}^1_{(y_0:y_1)}$ is induced on $Z_S$ by the map $\varphi_{|l_{Z}|}:Z_S\ra \mathbb{P}^1\times \mathbb{P}^1$. We observe that the map $\sigma_Z$ acts on the basis of this fibration, by switching the two copies of $\mathbb{P}^1$, and does not act on the fibers.

If we project \eqref{eq: ZS double cover  P1,4} to the first copy of $\mathbb{P}^1$, one obtains a fibration $Z_S\ra \mathbb{P}^1_{(x_0:x_1)}$ whose generic fibers are Calabi--Yau 3-folds which are double covers of $\mathbb{P}^1\times \mathbb{P}^1\times \mathbb{P}^1$ branched along a curve of multidegree $(4,4,4)$ which indeed splits into the union of 5 curves of multidegree $(1,0,0)$,$(1,0,0)$,$(1,0,0)$,$(1,0,0)$,$(0,4,4)$ respectively. This fibration is induced on $Z_S$ by the map $\varphi_{|l_{1,Z}|}:Z_S\ra\mathbb{P}^1$.\\

The map $f_Y:\varphi_{|l_{Y}|}:Y_S\ra \left(\mathbb{P}^{1}\times \mathbb{P}^1\right)^{(2)}$ is the fibration induced both by $\varphi_{l^{S^{[2]}}}:S^{[2]}\ra \mathbb{P}^2$ on the quotient $S^{(2)}/\iota_S^{(2)}$ and by $\varphi_{|l^{Z}|}:Z\ra \mathbb{P}^1\times\mathbb{P}^1$ on the quotient $Z_S/\sigma_Z$, by the definition of the divisors $l_Y$, $l_{S^{[2]}}$ and $l_{Z}$.

Generically the fiber of $f_Y$ are the Kummer surfaces, obtained either as quotients of the fibers of the fibration $f_{Z}:Z_S\ra\mathbb{P}^1_{(x_0:x_1)}\times \mathbb{P}^1_{(y_0:y_1)}$ by the involution acting on the fiber of the fibration as $((x_0:x_1),(y_0:y_1))\mapsto ((y_0:y_1),(x_0:x_1))$ or as the quotients of the fibers of the fibration $f_{S^{[2]}}\ra \mathbb{P}^2$ by the involution acting on the generic fiber, which is an Abelian surface, as the involution which sends each point in its opposite.\endproof

\begin{rem}{\rm In Proposition \ref{prop: maps S/iota=P1P1} we considered some divisors in the set $\mathcal{S}:=\{l_{S^{[2]}},l_{1,Z},l_Z,l_{1,Z}+m_{1,Z}=H_{1,Z}, l_Z+m_{1,Z},l_Y\}$. All these divisors are associated to fibrations defined on a variety and indeed they are not big, so one  we can not use the  Kawamata--Viehweg vanishing theorem to compute $h^0(D)$ if $D\in\mathcal{S}$. However by Proposition \ref{chi-formula} one can compute their $\chi(D)$ and one finds: 
$\chi(l_{S^{[2]}})=3$;
$\chi(l_{1,Z})=2$;
$\chi(l_{1,Z}+m_{1,Z})=\chi(H_{1,Z})=4$;
$\chi(l_{Z})=2;$
$\chi(l_Z+m_{1,Z})=8$
$\chi(l_Y)=2$.
We observe that the base of the fibration defined by the divisors $D\in\{l_{S^{[2]}},l_{1,Z},l_{1,Z}+m_{1,Z}=H_{1,Z}, l_Z+m_{1,Z}\}\subset\mathcal{S}$ has the property that it is either the projective space $\mathbb{P}^{\chi(D)-1}$ or it is embedded in the projective space $\mathbb{P}^{\chi(D)-1}$ by the Segre embedding. So one can expect that for these choices of $D$, $\chi(D)=h^0(D)$.
On the other hand this seems surely false if $D$ is either $l_Z$ or $l_Y$.
}\end{rem}

\subsection{$\iota_S$ is the elliptic involution on a generic elliptic fibration on the K3 surface $S$}\label{subsec: S/iota=F4}

In this case $S$ has an elliptic fibration with 24 fibers of type $I_1$ and $\NS(S)\simeq U$, $\iota_S$ restricts to the elliptic involution on each fiber of the elliptic fibration and $W\simeq \mathbb{F}_4$.
Generically $\NS(S)$ is generated by the class of a fiber, $F$, and by the class of the zero section, $O$, whose intersection properties are $F^2=0$, $O^2=-2$, $FO=1$. The divisor $F$ is nef, and it is such that $\varphi_{|F|}:S\ra\mathbb{P}^1$ is the elliptic fibration, so $F$ is not big. We will denote by $F_p$ the fiber of $f$ over the point $p$, i.e. $F_p\simeq f^{-1}(p)\subset S$. The involution $\iota_S$ fixes 2 curves, one is the rational curve $O$, section of the fibration, the other is a trisection, branched with multiplicity 2 on each singular fibers. It is denoted by $C$, it has genus 10 and its class in $\NS(S)$ is $6F+3O$.
We will denote by $H:=4F+2O$. The map $\varphi:S\ra\mathbb{P}^5$ is a $2:1$ map onto the image, which is a cone over a normal quartic rational curve.

\begin{prop}\label{prop: fibrations FX, S/iota=F4}
The map $\varphi_{|F_{S^{[2]}}|}:S^{[2]}\ra (\mathbb{P}^1)^{(2)}\simeq \mathbb{P}^2$ is a fibration whose generic fibers are products of two (non isogenous) elliptic curves.

The map $\varphi_{|F_Z|}:Z_S\ra \mathbb{P}^1\times \mathbb{P}^1\subset \mathbb{P}^3$ is a fibration whose generic fibers are Kummer surfaces of the product of two (non isogeneous) elliptic curves.

The map $\varphi_{|F_Y|}:Y_S\ra (\mathbb{P}^1)^{(2)}\simeq  \mathbb{P}^2$  is a fibration whose generic fibers are Kummer surfaces of the product of two (non isogeneous) elliptic curves. 

The fibration $\varphi_{|F_{1,Z}|}: Z_S\ra \mathbb{P}^1$ is a fibration whose generic fiber is a Calabi--Yau 3-fold of Borcea--Voisin type.
\end{prop}
\proof 
The 4-fold $S\times S$ admits a fibration $f\times f: S\times S\ra \mathbb{P}^1\times \mathbb{P}^1$ whose fiber over the point $(p,q)$ is the product $F_p\times F_q$ and it coincides with the map $\varphi_{|F_{1,S\times S}+F_{2,S\times S}|}:S\times S\ra \mathbb{P}^1\times \mathbb{P}^1$. So the generic fiber of $S\times S\ra \mathbb{P}^1\times \mathbb{P}^1$ is an Abelian surface, which is the product of two elliptic curves (generically non isogeneous). The section of the fibration $f$ defines a section of $f\times f$, passing through the zero of the Abelian surfaces. The involution $\iota_S\times \iota_S$ fixes this section and other three surfaces, which are two $3$-sections and one $9$-section. The involution $\iota_S\times \iota_S$ on $A\simeq E_p\times E_q$ acts sending each point in its inverse with respect to the group law. The automorphism $\sigma$ does not preserve the fibers of the fibration $f\times f$: it acts on the basis, switching the two copies of $\mathbb{P}^1\times \mathbb{P}^1$ and sending 
the fiber $F_
p\times F_q$ to the fiber $F_q\times F_p$. 
This allows to describe the fibrations induced by $\varphi_{|F_{S\times S}|}$ on the quotients of $S\times S$ as follows:
the hyperk\"ahler fourfold $S^{[2]}$ naturally admits a Lagrangian fibration $f^{[2]}:S^{[2]}\ra\PP^2$, whose fibers are generically the product of the corresponding fibers on the K3 surface $S$. Indeed, $f\times f$ is equivariant with respect to the action of the exchange $\sigma$; hence, we get an induced fibration $f^{(2)}$ of $S^{(2)} $ over $(\PP^1)^{(2)}\cong\PP^2$. The so-called natural Lagrangian fibration $f^{[2]}$ is the composition $f^{(2)}\circ \beta_{\Fix(\sigma)}$, where $\beta_{\Fix(\sigma)}$ is the resolution $S^{[2]}\ra S\times S/\sigma$ and it coincides with $\varphi_{|F_{S^{[2]}}|}$. Let $\Delta_S\subset S\times S$ be the diagonal and let $E_{\Delta_S}$ be the exceptional divisor on $S^{[2]}$. By \cite[Theorem 1]{Fu}, $f^{[2]}(E)$ is one-dimensional, so that the generic fiber of $f^{[2]}$ does not intersect $E$ and is isomorphic to the generic fiber of $f^{(2)}$, which is the common image of $F_p\times F_q$ and of $F_q\times F_p$, still isomorphic to  $F_p\times F_q$.

The automorphism $\iota_S$ does not act on the basis of the fibration $f:S\ra \mathbb{P}^1$, so the basis of the fibration induced by $f\times f$ on $(S\times S)/(\iota_S\times \iota_S)$ is $\mathbb{P}^1\times \mathbb{P}^1$. The fiber over the generic point $(p,q)\in\mathbb{P}^1\times \mathbb{P}^1$ of the fibration $(S\times S)/(\iota_S\times \iota_S)\ra\mathbb{P}^1\times \mathbb{P}^1$ are the quotients of the Abelian surfaces $F_p\times F_q$, fibers of $f\times f:S\times S\ra \mathbb{P}^1\times \mathbb{P}^1$, by the involution which sends each point in its inverse with respect to the group law. The fibration $(S\times S)/(\iota_S\times \iota_S)\ra\mathbb{P}^1\times \mathbb{P}^1$ extends to a fibration $Z_S\ra \mathbb{P}^1\times \mathbb{P}^1$ whose fiber over a generic point $(p,q)$ is the  Kummer surfaces $\mathrm{Km}(F_p\times F_q)$ desingularizations of the singular fibers of $(S\times S)/(\iota_S\times \iota_S)\ra\mathbb{P}^1\times \mathbb{P}^1$. By construction this fibration is given by the map $\
varphi_
{|F_Z|}:Z_S\ra \mathbb{P}^1\times \mathbb{P}^1$. The strict transform of $O\times O$ is a  section of the fibration and it meets the generic fiber in a rational curve. We recall that a Kummer surface $\mathrm{Km}(A)$ contains 16 disjoint rational curves which are in $1:1$ 
correspondence with the 2-torsion points in $A$. The "zero section" of the fibration $Z_S\ra \mathbb{P}^1\times \mathbb{P}^1$ is the section which meets the smooth fibers (which are a Kummer surfaces $\mathrm{Km}(F_p\times F_q)$) in the rational point which correspond to the 0 of the Abelian surface $F_p\times F_q$. Similarly the strict transform of  $O\times C$ (resp. $C\times O$, $C\times C$) meets the smooth fibers in 3 (resp. $3, 9$) rational curves, corresponding to other 3 (resp. 3,9) points of order 2 on $F_p\times F_q$.

The automorphism $\sigma_Z$ acts on the basis of this fibration $\varphi_{|F_Z|}:Z_S\ra \mathbb{P}^1\times \mathbb{P}^1$. Outside of the diagonal of $\mathbb{P}^1\times \mathbb{P}^1$, $\sigma_Z$ identifies two fibers (the fiber $\mathrm{Km}(F_p\times F_q)$ with the fiber $\mathrm{Km}(F_q\times F_p)$). This identification sends the 2-torsion point of $F_p\times F_q$ to the one of $F_q\times F_p$. So on $Z_S/\sigma_{Z}$ we have a fibration whose generic fibers are Kummer surfaces. In particular the map $\varphi_{|F_Y|}:S\ra (\mathbb{P}^1)^{(2)}\simeq \mathbb{P}^2$ defines a fibration whose generic fibers are Kummer surfaces of the product of 2 non isogenous elliptic curves. We observe that the generic fibers of this fibration are not isomorphic, but all of them are polarized with the same lattice. This fibration has a section, induced by the section of the fibration $Z_S\ra \mathbb{P}^1\times \mathbb{P}^1$.

The fibration $F_{1,Z}:Z_S\ra \mathbb{P}^1_{\tau}$ exhibits $Z_S$ as a fibration in Calabi--Yau 3-folds and the fiber over a generic point $\tau\in\mathbb{P}^1$  is the Borcea-Voisin of $S\times F_{\tau}$, i.e. it is the desingularization of $(S\times F_{\tau})/(\iota_S\times\iota_{F_{\tau}})$, where $\iota_{F_\tau}$ is the elliptic involution on the elliptic curve $F_t$ and it is the restriction of $\iota_S$ to the fiber $F_\tau$ of the fibration $S\ra\mathbb{P}^1$. This easily follows by our construction but can also be written explicitly by using the equations of the fibrations. Let us write the equation of $S\ra\mathbb{P}^1_t$ as $y^2=x^3+a(\tau)x+b(\tau)$ (where $a(\tau)$ and $b(\tau)$ are polynomials of degree 8 and 12 respectively). Then the second copy of $S$ in $S\times S$ has an equation of the type $v^2=u^3+a(s)u+b(s)$, which can be written as $v^2=(u^3+a(s:t)uz^2+b(s:t)z^3)z$. This exhibits this second copy of $S$ as double cover of the Hirzebruch surface $\mathbb{F}_4$ with variables $(s:t:x:z)$ 
(cf. \cite{CG}). Since $\iota_S\times \iota_S$ changes the sign of $y$ and $v$, the functions  $Y:=yv^3$, $X:=xv^2$, $\tau$, $t$, $s$ are invariant, so with these coordinates the equations for $(S\times S)/(\iota_S\times \iota_S)$ are 
$$Y^2=X^3+a(\tau)X\left(x^3+a(s:t)xz^2+b(s:t)z^3\right)^2z^2+b(\tau)\left(x^3+a(s:t)xz^2+b(s:t)z^3\right)^3z^3.$$
For generic choices of $\tau=\overline{\tau}$, this equation is the equation of the Borcea-Voisin Calabi--Yau 3-folds given in \cite[Section 4.4]{CG}. 
\endproof

\begin{rem}{\rm The fixed locus of $\iota_S$ on $S$ is given by $O\cap C$ and the class of the fixed locus is $O+6F+3O=6F+4O$. This allows to compute $\chi(F_Z)=\chi(F_{1,S\times S}+F_{2,S\times S})/2+h_{\Sigma,Z}/16=2+2=4$. The base of the fibration $\varphi_{|F_Z|}:Z_S\ra\mathbb{P}^1\times \mathbb{P}^1$ is embedded in $\mathbb{P}^{\chi(F_Z)-1}$ by the Segre embedding.
Similarly, one computes $\chi(F_Y)=2+1=3$ and $\chi(F_{S^{[2]}})=3$ and $\chi(F_1,Z)$, and one observes that the bases of the fibrations $\varphi_{|F_Y|}$, $\varphi_{|F_{S^{{2}}}|}$ and $\varphi_{|F_{1,Z}|}$ are again $\mathbb{P}^{\chi-1}$. This suggests that $\chi$ is equal to $h^0$ for all the divisors involved in Proposition \ref{prop: fibrations FX, S/iota=F4}.}\end{rem}

\begin{prop}
The map $\varphi_{|H_{S^{[2]}}|}:S^{[2]}\ra \left(\mathbb{P}^5\right)^{(2)}\subset \mathbb{P}^{20}$ is a $4:1$ map, where the inclusion $\left(\mathbb{P}^5\right)^{(2)}\subset \mathbb{P}^{20}$ is given by $$i_5:\lbrace(x_0:\ldots :x_5),(y_0:\ldots y_5)\rbrace\mapsto (x_0y_0: x_0y_1+x_1y_0:\ldots :x_iy_j+x_jy_i: \ldots x_5y_5),$$ with $i<j$.

The map $\varphi_{|F_Z|}:Z_S\ra \mathbb{P}^5\times \mathbb{P}^5\subset \mathbb{P}^{35}$ is a $2:1$ map and $\mathbb{P}^5\times \mathbb{P}^5\subset \mathbb{P}^{35}$ is given by the Segre embedding.

The map $\varphi_{|F_Y|}:Y_S\ra \mathbb{P}^{21}$  is a $2:1$ map and the inclusion $\left(\mathbb{P}^5\right)^{(2)}\subset \mathbb{P}^{20}$ is given by $i_5$.
\end{prop}
The proof is analogous to the one of Proposition \ref{prop: maps if NS=2}.

\bibliographystyle{amsplain}
\bibliography{QuotientsHK}
\section{Appendix: Hodge numbers of Calabi--Yau four-folds constructed}\label{Appendix}
Here we collect the Hodge numbers of the Calabi--Yau four-folds constructed in Sections \ref{sec: quotients of hyperkahler} and
\ref{subsec: Y_S}.
In Sections \ref{subsec: generalized Kummer} and \ref{subsec Beauville non natural} we constructed Calabi--Yau four-folds with Hodge numbers:
$$
\begin{array}{rr}
\begin{array}{||c|c|c|c||}
\hline
h^{1,1}&h^{2,1}&h^{3,1}&h^{2,2}\\\hline
    9&  8&  5&  75\\ \hline
    6&  4&  4&  68\\ \hline
    5&  3&  4&  66\\ \hline
\end{array}
\mbox{ and }
\begin{array}{||c|c|c|c||}
\hline
h^{1,1}&h^{2,1}&h^{3,1}&h^{2,2}\\\hline
    2&  0&  65& 312\\ \hline
\end{array}
\end{array}$$
respectively.

In Section \ref{subsec: Y_S} we constructed Calabi--Yau varieties $Y_S$ starting from a K3 surface $S$ with a non symplectic involution $\iota_S$ whose fixed locus contains $N$ curves $C_i$ such that $N'=\sum_{i=0}^Ng(C_i)$. In the following table we list  the Hodge numbers of $Y_S$ in terms of $(N,N')$:
$$
\begin{array}{rr}
\begin{array}{||c|c|c|c|c|c||}
\hline
N&N'&h^{1,1}&h^{2,1}&h^{3,1}&h^{2,2}\\\hline
0&  0&  12& 0&  10& 132\\ \hline
1&  0&  14& 0&  9&  136\\ \hline
1&  1&  13& 1&  10& 134\\ \hline
1&  2&  12& 2&  12& 136\\ \hline
1&  3&  11& 3&  15& 142\\ \hline
1&  4&  10& 4&  19& 152\\ \hline
1&  5&  9&  5&  24& 166\\ \hline
1&  6&  8&  6&  30& 184\\ \hline
1&  7&  7&  7&  37& 206\\ \hline
1&  8&  6&  8&  45& 232\\ \hline
1&  9&  5&  9&  54& 262\\ \hline
1&  10& 4&  10& 64& 296\\ \hline
2&  0&  17& 0&  8&  144\\ \hline
2&  1&  16& 2&  9&  140\\ \hline
2&  2&  15&4&   11& 140\\ \hline
2&  3&  14& 6&  14& 144\\ \hline
2&  4&  13& 8&  18& 152\\ \hline
2&  5&  12& 10& 23& 164\\ \hline
2&  6&  11& 12& 29& 180\\ \hline
2&  7&  10& 14& 36& 200\\ \hline
2&  8&  9&  16& 44& 224\\ \hline
2&  9&  8&  18& 53& 252\\ \hline
2&  10& 7&  20& 63& 284\\ \hline
3&  0&  21& 0&  7&  156\\ \hline
3&  1&  20& 3&  8&  150\\ \hline
3&  2&  19& 6&  10& 148\\ \hline
3&  3&  18& 9&  13& 150\\ \hline
3&  4&  17& 12& 17& 156\\ \hline
3&  5&  16& 15& 22& 166\\ \hline
3&  6&  15& 18& 28& 180\\ \hline
3&  7&  14& 21& 35& 198\\ \hline
4&  0&  26& 0&  6&  172\\ \hline
4&  1&  25& 4&  7&  164\\ \hline
\end{array}
&
\begin{array}{||c|c|c|c|c|c||}
\hline
N&N'&h^{1,1}&h^{2,1}&h^{3,1}&h^{2,2}\\\hline
4&  2&  24& 8&  9&  160\\ \hline
4&  3&  23& 12& 12& 160\\ \hline
4&  4&  22& 16& 16& 164\\ \hline
4&  5&  21& 20& 21& 172\\ \hline
4&  6&  20& 24& 27& 184\\ \hline
5&  0&  32& 0&  5&  192\\ \hline
5&  1&  31& 5&  6&  182\\ \hline
5&  2&  30& 10& 8&  176\\ \hline
5&  3&  29& 15& 11& 174\\ \hline
5&  4&  28& 20& 15& 176\\ \hline
5&  5&  27& 25& 20& 182\\ \hline
5&  6&  26& 30& 26& 192\\ \hline
6&  0&  39& 0&  4&  216\\ \hline
6&  1&  38& 6&  5&  204\\ \hline
6&  2&  37& 12& 7&  196\\ \hline
6&  3&  36& 18& 10& 192\\ \hline
6&  4&  35& 24& 14& 192\\ \hline
6&  5&  34& 30& 19& 196\\ \hline
6&  6&  33& 36& 25& 204\\ \hline
7&  0&  47& 0&  3&  244\\ \hline
7&  1&  46& 7&  4&  230\\ \hline
7&  2&  45& 14& 6&  220\\ \hline
7&  3&  44& 21& 9&  214\\ \hline
8&  0&  56& 0&  2&  276\\ \hline
8&  1&  55& 8&  3&  260\\ \hline
8&  2&  54& 16& 5&  248\\ \hline
9&  0&  66& 0&  1&  312\\ \hline
9&  1&  65& 9&  2&  294\\ \hline
9&  2&  64& 18& 4&  280\\ \hline
10& 0&  77& 0&  0&  352\\ \hline
10& 1&  76& 10& 1&  332\\ \hline
10& 2&  75& 20& 3&  316\\ \hline
\end{array}
\end{array}
$$

One can directly check that there are no mirror pairs in the previous table, except for the self-mirror Calabi--Yau $Y_S$ associated to the values $N'=N+1$ for $N=1,\ldots ,5$.

\end{document}